\documentclass[a4paper, fleqn]{amsart}
\usepackage{amsmath,amsfonts,amssymb,amsthm}
\usepackage{mathtools}
\usepackage{mathrsfs}
\usepackage{pdfpages}
\usepackage{url}
\usepackage{color,xcolor}
\usepackage{caption}
\usepackage{enumerate}
\usepackage{afterpage}
\usepackage{xargs,twoopt}
\usepackage{hyperref}
\usepackage[toc,page]{appendix}
\usepackage{spverbatim}
\hypersetup{colorlinks = true, linkbordercolor = {white}}
\usepackage{aliascnt}
\usepackage{bbold}
\usepackage{dsfont}
\usepackage[cal=boondoxo]{mathalfa}
\usepackage{bussproofs}
\usepackage{graphicx}
\usepackage[retainorgcmds]{IEEEtrantools}
\usepackage{stmaryrd}
\usepackage{tensor}
\usepackage{textcomp}
\usepackage{tikz}
\usepackage{units}
\usepackage[all]{xy}
\DeclareSymbolFont{extraitalic}      {U}{zavm}{m}{it}
\DeclareMathSymbol{\Qoppa}{\mathord}{extraitalic}{161}
\DeclareMathSymbol{\qoppa}{\mathord}{extraitalic}{162}
\DeclareMathSymbol{\Stigma}{\mathord}{extraitalic}{167}
\DeclareMathSymbol{\Sampi}{\mathord}{extraitalic}{165}
\DeclareMathSymbol{\sampi}{\mathord}{extraitalic}{166}
\DeclareMathSymbol{\stigma}{\mathord}{extraitalic}{168}

\newcommand{\id}{\text{id}}
\newcommand{\ceA}{\mathcal{A}}

\newcommand{\ceC}{\mathcal{C}}
\newcommand{\ceD}{\mathcal{D}}

\newcommand{\ceI}{\mathcal{I}}
\newcommand{\ceJ}{\mathcal{J}}
\newcommand{\ceK}{\mathcal{K}}
\newcommand{\ceL}{\mathcal{L}}
\newcommand{\ceM}{\mathcal{M}}
\newcommand{\ceN}{\mathcal{N}}

\newcommand{\ceQ}{\mathcal{Q}}
\newcommand{\ceR}{\mathcal{R}}

\newcommand{\cB}{\mathscr{B}}
\newcommand{\cC}{\mathscr{C}}
\newcommand{\cD}{\mathscr{D}}

\newcommand{\cH}{\mathscr{H}}
\newcommand{\cI}{\mathscr{I}}
\newcommand{\cJ}{\mathscr{J}}

\newcommand{\cL}{\mathscr{L}}
\newcommand{\cM}{\mathscr{M}}

\newcommand{\cR}{\mathscr{R}}

\newcommand{\cU}{\mathscr{U}}

\newcommand{\cX}{\mathscr{X}}
\newcommand{\Lg}{\mathfrak{g}}

\newcommand{\LA}{\mathfrak{A}}

\newcommand{\LR}{\mathfrak{R}}
\newcommand{\LX}{\mathfrak{X}}

\newcommand{\cCat}{\textbf{Cat}}

\newcommand{\mbm}{\mathbf{m}}

\newcommand{\mbC}{\mathbf{C}}

\newcommand{\mbG}{\mathbf{G}}

\newcommand{\mbM}{\mathbf{M}}
\newcommand{\mbN}{\mathbf{N}}
\newcommand{\mbP}{\mathbf{P}}

\newcommand{\bbon}{\mathbb{1}}
\newcommand{\bbk}{\mathbb{k}}

\newcommand{\bbZ}{\mathbb{Z}}

\newcommand{\tti}{\mathtt{i}}
\newcommand{\ttj}{\mathtt{j}}
\newcommand{\ttk}{\mathtt{k}}
\newcommand{\ttl}{\mathtt{l}}

\newcommand{\ttI}{\mathtt{I}}

\newcommand{\ttL}{\mathtt{L}}
\newcommand{\ttP}{\mathtt{P}}

\newcommand{\ul}[1]{\underline{#1}}
\newcommand{\ol}[1]{\overline{#1}}

\DeclareMathOperator{\End}{End}

\DeclareMathOperator{\Hom}{Hom}

\DeclareMathOperator{\Len}{Len}

\DeclareMathOperator{\Mod}{Mod}

\DeclareMathOperator{\add}{add}

\DeclareMathOperator{\coev}{coev}

\DeclareMathOperator{\comod}{comod}
\DeclareMathOperator{\bcomod}{\textbf{comod}}

\DeclareMathOperator{\ev}{ev}

\DeclareMathOperator{\inj}{inj}
\DeclareMathOperator{\binj}{\textbf{inj}}

\DeclareMathOperator{\modd}{mod}

\DeclareMathOperator{\proj}{proj}

\DeclareMathOperator{\rad}{rad}

\EnableBpAbbreviations
\author{James Macpherson}
\newtheorem{thm}{Theorem}[section]
\newaliascnt{lem}{thm}
\newtheorem{lem}[lem]{Lemma}
\aliascntresetthe{lem}
\newaliascnt{psn}{thm}
\newtheorem{psn}[psn]{Proposition}
\aliascntresetthe{psn}
\newaliascnt{cor}{thm}
\newtheorem{cor}[cor]{Corollary}
\aliascntresetthe{cor}
\newaliascnt{que}{thm}

\aliascntresetthe{que}
\newaliascnt{conj}{thm}

\aliascntresetthe{conj}
\theoremstyle{definition}
\newaliascnt{ex}{thm}

\aliascntresetthe{ex}
\newaliascnt{exs}{thm}

\aliascntresetthe{ex}
\newaliascnt{def}{thm}
\newtheorem{defn}[def]{Definition}
\aliascntresetthe{def}
\newaliascnt{nota}{thm}
\newtheorem{nota}[nota]{Notation}
\aliascntresetthe{nota}
\newaliascnt{rmk}{thm}

\aliascntresetthe{rmk}
\newaliascnt{rmks}{thm}

\aliascntresetthe{rmks}

\usepackage{stmaryrd}
\title[2-Representation Theory Of Locally Finitary 2-Categories]{Extension Of The 2-Representation Theory Of Finitary 2-Categories To Locally (Graded) Finitary 2-Categories}

\begin{document}
\bibliographystyle{alpha}

\begin{abstract}
	We extend the 2-representation theory of finitary 2-categories to certain 2-categories with infinitely many objects, called locally finitary 2-categories, and extend the classical classification results of simple transitive 2-representations of weakly fiat 2-categories to this environment. We also consider locally finitary 2-categories and 2-representations with a grading, and prove that the associated coalgebra 1-morphisms have a homogeneous structure. We use these results to classify (graded) simple transitive 2-representations of certain classes of cyclotomic 2-Kac--Moody algebras.
\end{abstract}
\maketitle
\section{Introduction}
	The study of the 2-representation theory of finitary and fiat 2-categories, pioneered by Mazorchuk and Miemietz in \cite{mazorchuk2011cell} through \cite{mazorchuk2016isotypic} and further explored in various other works, e.g. \cite{mackaay2016simple} and \cite{chan2016diagrams}, is a powerful new tool in representation theory. Important applications of the theory include certain quotients of 2-Kac--Moody algebras (see \cite{mazorchuk2015transitive}) and Soergel bimodules (see for example \cite{mackaay20192}).\par

However, while powerful, the setup used to date in this theory has multiple restrictions, primarily relating to finiteness conditions. Specifically, the theory considers 2-categories which have only finitely many objects and whose hom-categories have finitely many isomorphism classes of indecomposable 1-morphisms and finite-dimensional spaces of 2-morphisms. The relaxation of these restrictions would enable the study of a much wider class of examples using techniques analogous to those for 2-representations of finitary 2-categories.\par

This paper is an initial step in that direction, focussing on the relaxation to countably many objects in the 2-categories, giving `locally finitary' 2-categories. While this is a comparatively mild generalisation, it already enables the study of multiple interesting examples that were previously inaccessible, including a much wider class of quotients of 2-Kac--Moody algebras. For an examination of relaxing finiteness conditions for 1- and 2-morphisms, see the companion paper \cite{macpherson20212representations} by the same author.\par

In the course of this paper, we give the generalisation of multiple finitary results to the locally finitary case. Of specific note, in \autoref{MMMTL47} we construct for any transitive 2-representation of a locally weakly fiat 2-category an equivalent `internal' 2-representation of comodule 1-morphism categories, analagously to major results in \cite{ostrik2003module} and \cite{mackaay2016simple}. In \autoref{LocWFBig}, a generalisation of the primary result in \cite{mazorchuk2015transitive}, we further classify all simple transitive 2-representations of strongly regular locally weakly fiat 2-categories as being equivalent to cell 2-representations. We then utilise the latter result to classify all simple transitive 2-representations of cyclotomic 2-Kac--Moody algebras.\par

This paper also considers a generalisation of the above setup to the case where the locally finitary 2-categories have an additional graded structure. In this setup, we use the construction of a `degree zero' sub-2-category to show in \autoref{AN0EqAN} that the previously constructed internal 2-representations associated to a transitive 2-representation can be viewed as a `degree zero' construction in a canonical fashion. We use this result by considering again cyclotomic 2-Kac--Moody algebras, demonstrating in \autoref{Gr2KMClass} that any simple transitive 2-representation is in fact a graded 2-representation.\par

The structure of the paper is as follows. In \autoref{LocFinSec} we give the initial definitions for locally finitary 2-categories and their 2-representations, as well as various related notions used in the paper. We also give some minor results demonstrating that the cell structure of the 2-category is highly analogous in this generalisation. In \autoref{MMMTGen} we generalise various results from \cite{mackaay2016simple}, leading up to \autoref{MMMTL47} as well as \autoref{MM5ByMMMT}, which classifies the simple transitive 2-representations for locally finitary 2-categories associated to certain infinite dimensional algebras.\par

\autoref{MMXGen} considers instead generalisations of various results in the series of papers \cite{mazorchuk2011cell} through \cite{mazorchuk2016isotypic}, particularly the former paper and \cite{mazorchuk2015transitive}. The eventual goal of this section is \autoref{LocWFBig}. \autoref{2KMApp} presents an application of this result by demonstrating that cyclotomic 2-Kac--Moody algebras of given weights are locally weakly fiat 2-categories, and thus submit to the aforementioned theorem.\par

In the last section, we examine the further generalisation to locally restricted $G$-finitary 2-categories for some countable abelian group $G$. We construct a degree zero 2-category associated to such a 2-category, and use it to construct a degree zero coalgebra 1-morphism for a given graded transitive 2-representation of the original 2-category. This setup allows us to prove \autoref{AN0EqAN}. Finally, we apply this to the cyclotomic 2-Kac--Moody categories of given weights, showing that their cell 2-representations are all graded simple transitive 2-representations, leading to \autoref{Gr2KMClass}.
\subsection*{Acknowledgements}
The author would like to thank Vanessa Miemietz for her support and supervision during the PhD from whence this material was derived. In addition, the author would like to thank the referees for the paper, who provided many substantial and helpful suggestions that notably improved the paper.

\section{Locally Finitary 2-Categories}\label{LocFinSec}

\subsection{Initial Definitions}

Let $\bbk$ be an algebraically closed field.

\begin{defn}\label{TarCats} We denote by $\LA_\bbk$ the 2-category whose objects are small $\bbk$-linear idempotent complete categories, whose 1-morphisms are $\bbk$-linear additive functors and whose 2-morphisms are natural transformations. We further denote by $\LA_{\bbk}^f$ the full sub-2-category of $\LA_{\bbk}$ with objects those categories $\ceA$ such that $\ceA$ has only finitely many isomorphism classes of indecomposable objects, and such that $\dim\Hom_{\ceA}(\tti,\ttj)<\infty$ for all objects $\tti,\ttj\in\ceA$. We call objects in this 2-category \emph{finitary categories}. We finally denote by $\LR_{\bbk}$ as the full sub-2-category of $\LA_{\bbk}$ whose objects are equivalent to $A$-mod for some finite dimensional associative $\bbk$-algebra $A$.\end{defn}

\begin{defn} A 2-category $\cC$ is \emph{locally finitary} over $\bbk$ when it has countably many objects, $\cC(\tti,\ttj)\in\LA_{\bbk}^f$ for all objects $\tti, \ttj\in\cC$, horizontal composition is additive and $\bbk$-linear and the identity 1-morphism $\bbon_\tti$ is indecomposable for all $\tti$.\end{defn}

\begin{nota}For the duration of this paper, we notate 1-morphisms as $F,G,\dots$ and 2-morphisms as $\alpha,\beta,\dots$. We denote the composition of 1-morphisms $F$ and $G$ as either $F\circ G$ or $FG$, and we notate horizontal composition of 2-morphisms $\alpha$ and $\beta$ as $\alpha\circ_H\beta$ and vertical composition as $\alpha\circ_V\beta$.\end{nota}

We note that if $\cC$ has only finitely many objects, then this is the definition of a finitary 2-category first given in \cite{mazorchuk2011cell}. Since all the applications of interest known to the author utilise at most countably many objects, this paper restricts all definitions to this case. It seems likely that many of the results will still apply in larger cases; however, the author has not confirmed this.

For the rest of this subsection, we let $\cC$ denote a locally finitary 2-category.

\begin{defn} Given a finite set of objects $\ul{\tti}=\{\tti_1,\dots,\tti_n\}$ in $\cC$, we denote the sub-2-category generated by the $\tti_j$ (i.e. with objects $\tti_1,\dots,\tti_n$, and with hom-categories $\cC(\tti_j,\tti_k)$ for all $j$ and $k$) by $\cC_{\ul{\tti}}$, and call it a \emph{full finitary sub-2-category}.\end{defn}

In this paper we will sometimes refer to proofs (generally in other papers) as having `local' proofs - we use this to indicate that the proof works by proving the result for any arbitrary set of objects, 1-morphisms and 2-morphisms that belong to a full finitary sub-2-category of the 2-category under consideration. The proof can thus be applied directly to the locally finitary setup without having to adjust anything beyond notation.

\begin{defn} We say $\cC$ is \emph{locally weakly fiat} if it has a weak object-preserving anti-autoequivalence $-^*$ and if for any 1-morphism $F\in\cC(\tti,\ttj)$ there exist 2-morphisms $\alpha:F\circ F^*\to\bbon_\ttj$ and $\beta:\bbon_\tti\to F^*\circ F$ such that $(\alpha\circ_H \id_F)\circ_V (\id_F\circ_H\beta)=\id_F$ and $(\id_{F^*}\circ_H \alpha)\circ_V(\beta\circ_{H}\id_{F^*})=\id_{F^*}$ hold. We let $\tensor[^*]{(-)}{}$ denote the inverse of this weak anti-autoequivalence. If $*$ is a weak anti-involution, we say that $\cC$ is a \emph{locally fiat} 2-category.\end{defn}

\begin{defn} An \emph{additive} (respectively \emph{finitary}, \emph{abelian}) \emph{2-representation} of $\cC$ is a strict 2-functor $\mathbf{M}:\cC\to\LA_{\bbk}$ (resp. $\mathbf{M}:\cC\to\LA_{\bbk}^f$, $\mathbf{M}:\cC\to\LR_{\bbk}$).
\end{defn}

Given a 2-representation $\mbM$ of $\cC$, we use the notation $\ceM:=\coprod_{\tti\in\cC} \mbM(\tti)$ for the corresponding category.

\begin{defn} The \emph{$\tti$th principal 2-representation} $\mbP_\tti:\cC\to\LA_\bbk$ of $\cC$ is defined on objects $\ttj\in\cC$ as $\mbP_\tti(\ttj)=\cC(\tti,\ttj)$, on 1-morphisms $F$ as $\mbP_\tti(F)=F\circ-$, and on 2-morphisms $\alpha$ as $\mbP_\tti(\alpha)=\alpha\circ_H-$. \end{defn}

Since $\cC$ is a locally finitary 2-category, this is a finitary 2-representation.
	
	\subsection{Cells, Ideals And Multisemigroups}
	
\begin{defn} Given a $\bbk$-linear 2-category $\cC$, we define a \emph{left 2-ideal} $\cI$ of $\cC$ to have the same objects as $\cC$, and for each pair $\tti,\ttj$ of objects an ideal $\cI(\tti,\ttj)$ of the 1-category $\cC(\tti,\ttj)$ which is stable under left horizontal multiplication with 1- and 2-morphisms of $\cC$. We similarly define \emph{right 2-ideals} and \emph{two-sided 2-ideals}. We call the latter simply \emph{2-ideals}.\end{defn}

Following standard constructions, given a 2-category $\cC$ and a 2-ideal $\cI$ of $\cC$, we define the \emph{quotient 2-category} $\cC/\cI$ as follows: the objects and 1-morphisms of $\cC/\cI$ are the same as those of $\cC$ and the hom-sets between 1-morphisms are defined as $$\Hom_{\cC/\cI(\tti,\ttj)}(F,G):=\Hom_{\cC(\tti,\ttj)}(F,G)/\Hom_{\cI(\tti,\ttj)}(F,G).$$

\begin{defn} Given a locally finitary 2-category $\cC$ and a 2-representation $\mbM$ of $\cC$, an \emph{ideal} $\cI$ of $\mbM$ is a collection of ideals $\cI(\tti)\subseteq \mbM(\tti)$ which is closed under the action of $\cC$ in that for any morphism $f\in\cI(\tti)$ and any 1-morphism $F\in\cC$, $\mbM(F)(f)$ is a morphism in $\cI$ if it is defined.\end{defn}

\begin{defn} Let $\cC$ be a locally finitary 2-category. We denote by $S(\cC)$ the multisemigroup of isomorphism classes of indecomposable 1-morphisms of $\cC$ (we denote the isomorphism class of $F$ by $[F]$), with the operation given by $$[F]*[G]=\{[H]\in S(\cC)|H\text{ is a direct summand of }FG\}.$$ We often abuse notation and write $F$ for the isomorphism class $[F]$.\end{defn}

The Green's relations for multisemigroups were first defined in \cite{kudryavtseva2012multisemigroups}, based on the original definition for semigroups in \cite{green1951structure}. To recall, if $(S,*)$ is a multisemigroup with $x\in S$, then the \emph{principal left ideal} of $x$ is the set $L_x=Sx\cup\{x\}$, the \emph{right principal ideal} is $R_x=xS\cup\{x\}$, and the \emph{two-sided principal ideal} is $J_x=SxS\cup Sx\cup xS\cup\{x\}$. This gives rise to equivalence relations $\cL$, $\cR$ and $\cJ$ where e.g. $x\sim_\cL y$ if $L_x=L_y$. The equivalence classes of these relations are called \emph{$\cL$-, $\cR$-} and \emph{$\cJ$-cells} respectively. There are also two further Green's relations. One is $\cH:=\cL\cap\cR$, and the other is $\cD$, defined as the join of $\cL$ and $\cR$ in the poset of equivalence relations - that is, it is the smallest equivalence relation to contain both $\cL$ and $\cR$. In semigroups, it is always the case that $\cD=\cL\circ\cR=\cR\circ\cL$ (see e.g. \cite[p. 219]{lawson2003finite}). However, in multisemigroups, this is not in general true. The best we can say is the following:

\begin{lem}\label{DPowersOfLR} $\cD=\bigcup_{i\in\bbZ^+} (\cL\circ\cR)^{\circ i}$.
	\begin{proof} Since $\cL,\cR\subseteq\cL\circ\cR$, and since $\cD$ is the join of $\cL$ and $\cR$, it follows that $\cD\subseteq \bigcup_{i\in\bbZ^+} (\cL\circ\cR)^{\circ i}$. However, if $\cL$ and $\cR$ are contained in an equivalence relation $M$, then it follows from transitivity that $(\cL\circ\cR)^{\circ i}\subseteq M$ for all $i$. Thus $\bigcup_{i\in\bbZ^+} (\cL\circ\cR)^{\circ i}\subseteq \cD$, and the result follows.\end{proof}\end{lem}

We note that, even in semigroup theory, it is not always true that $\cD=\cJ$ in the infinite case, and we need to apply care when considering a locally finitary 2-category. That said, we are able to give useful results.\par

For the rest of the subsection, let $\cC$ be a locally finitary 2-category with multisemigroup of isomorphism classes of indecomposables $S(\cC)$, on which we have Green's relations $\cL_\cC$, $\cR_\cC$, $\cJ_\cC$, $\cD_\cC$ and $\cH_\cC$. For a full finitary sub-2-category $\cB$ of $\cC$, $S(\cB)\subseteq S(\cC)$. Let $\cL_\cB$, $\cR_\cB$, $\cJ_\cB$, $\cD_\cB$ and $\cH_\cB$ denote the Green's relations of $S(\cB)$, which we consider as equivalence relations on $S(\cB)$. We can extend these to equivalence relations on $S(\cC)$ by setting $\ol{\cX_\cB}=\cX_\cB\cup\Delta_{S(\cC)}$ for $\cX$ one of the Green's relations and $\Delta$ the diagonal equivalence relation. We first show that all the Green's relations on $\cC$ barring $\cH_\cC$ are determined by the Green's relations on the full finitary sub-2-categories.

\begin{psn}\label{cLbootleg} Let $\chi$ denote the set of full finitary sub-2-categories of $\cC$. Then $\bigcup_{\cB\in\chi} \ol{\cL_\cB}=\cL_\cC$. A similar result holds for $\cR_\cC$ and $\cJ_\cC$.
	
	\begin{proof} All three results have similar proofs - we give the proof for $\cL_\cC$. If we have $(F,G)\in\bigcup_{\cB\in\chi} \ol{\cL_\cB}$, either $(F,G)\in\Delta_{S(\cC)}$ and $F=G$ or there exists some $\cB\in\chi$ such that $(F,G)\in\cL_\cB$. If $F=G$ then $(F,G)\in\cL_\cC$. If $(F,G)\in\cL_\cB$ then there exist $H,K\in S(\cB)$ such that $F$ is a direct summand of $HG$ and $G$ is a direct summand of $KF$. But then it follows that $F\sim_{\cL_\cC} G$ and thus $(F,G)\in\cL_\cC$ and $\bigcup_{\cB\in\chi} \ol{\cL_\cB}\subseteq \cL_\cC$.\par
		
		Conversely, if $(F,G)\in\cL_\cC$, then there exist some $H,K\in S(\cC)$ such that $F$ is a direct summand of $HG$ and $G$ is a direct summand of $KF$. Now let $\cB$ be a full finitary sub-2-category of $\cC$ that contains $F,G,H$ and $K$. It follows that $(F,G)\in\cL_\cB$, and hence $(F,G)\in\bigcup_{\cB\in\chi} \ol{\cL_\cB}$. Thus $\cL_\cC\subseteq \bigcup_{\cB\in\chi} \ol{\cL_\cB}$, and $\cL_\cC=\bigcup_{\cB\in\chi} \ol{\cL_\cB}$ as required.\end{proof}\end{psn}

\begin{psn}\label{cDbootleg} In the same setup as above, $\bigcup_{\cB\in\chi} \ol{\cD_\cB}=\cD_\cC$.
	\begin{proof} If $(F,G)\in\cL_\cC$ then as by \autoref{cLbootleg} $\cL_\cC=\bigcup_{\cB\in\chi} \ol{\cL_\cB}$, without loss of generality $(F,G)\in\cL_\cB$ for some $\cB\in\chi$. By the definition of $\cD_\cB$, $(F,G)\in\cD_\cB$. Thus $(F,G)\in \bigcup_{\cB\in\chi} \ol{\cD_\cB}$, and hence $\cL_\cC\subseteq\bigcup_{\cB\in\chi} \ol{\cD_\cB}$. Similarly, $\cR_\cC\subseteq\bigcup_{\cB\in\chi} \ol{\cD_\cB}$. But then as $\cD_\cC$ is the join of $\cL_\cC$ and $\cR_\cC$, we must have that $\cD_\cC\subseteq \bigcup_{\cB\in\chi} \ol{\cD_\cB}$.\par
		
		For the opposite inclusion, let $(F,G)\in\bigcup_{\cB\in\chi} \ol{\cD_\cB}$. It follows from \autoref{DPowersOfLR} that $\cD_\cB=\bigcup_{i\in\bbZ^+} (\cL_\cB\circ\cR_\cB)^{\circ i}$. Therefore there exists some $n\in2\bbZ^+$ and $H_i\in S(\cB)$ for $0\leq i\leq n$ such that $$F=H_0\sim_{\cL_\cB} H_1\sim_{\cR_\cB} H_2\sim_{\cL_\cB}\dots\sim_{\cL_\cB} H_{n-1}\sim_{\cR_\cB} H_n=G.$$ Then by \autoref{cLbootleg}, $$(H_i,H_{i+1})\in\cL_\cB\Rightarrow (H_i,H_{i+1})\in\cL_\cC,$$ and $$(H_i,H_{i+1})\in\cR_\cB\Rightarrow (H_i,H_{i+1})\in\cR_\cC.$$ Therefore $(H_i,H_{i+1})\in\cD_\cC$ for all $i$. But as $\cD_\cC$ is an equivalence relation it follows that $(F,G)\in\cD_\cC$. Thus $\bigcup_{\cB\in\chi} \ol{\cD_\cB}\subseteq \cD_\cC$ and the result follows. \end{proof}\end{psn}
	
	We can also give a useful result for sufficiently nice $\cJ$-cells of $\cC$:

\begin{thm}\label{LFJIsD} Let $\ceJ$ be a $\cJ$-cell of $\cC$ such that every $\cH$-cell of $\ceJ$ is non-empty. Let $\cL_\ceJ$, $\cR_\ceJ$, $\cD_\ceJ$ and $\cJ_\ceJ$ denote the restrictions of the Green's relations of $\cC$ to $\ceJ$. Then $\cL_\ceJ\circ\cR_\ceJ=\cR_\ceJ\circ\cL_\ceJ=\cD_\ceJ=\cJ_\ceJ$.
	\begin{proof} The proof of \cite[Proposition 28 b)]{mazorchuk2011cell} is local and generalises immediately, proving that $\cL_\ceJ\circ\cR_\ceJ=\cR_\ceJ\circ\cL_\ceJ=\cJ_\ceJ$. Finally, it is immediate from the definitions that $\cL_\ceJ\circ\cR_\ceJ\subseteq\cD_\ceJ\subseteq\cJ_\ceJ$, and thus the remaining equality follows directly from the prior equalities.  \end{proof}\end{thm}

 We note that we have corresponding partial orders $\leq_\cL$, $\leq_\cR$ and $\leq_\cJ$ on $S(\cC)$ such that e.g. $(F,G)\in\cL$ if and only if $F\leq_\cL G$ and $G\leq_\cL F$, where $F\leq_\cL G$ if there is some 1-morphism $H$ such that $G$ is a direct summand of $HF$, with similar definitions for $\leq_\cR$ and $\leq_\cJ$.
 
\begin{defn} A $\cJ$-cell $\ceJ$ of $\cC$ is \emph{strongly regular} if left cells in $\ceJ$ are incomparable under the left order and if for any left cell $\ceL\subseteq\ceJ$ and right cell $\ceR\subseteq\ceJ$, $\ceL\cap\ceR$ contains a unique isomorphism class of indecomposable 1-morphisms. If only the first condition holds then $\ceJ$ is called \emph{regular}.\end{defn}

\begin{defn} Given a 2-representation $\mbM$ of $\cC$ and a collection of objects $\{X_j\}_{j\in J}$ in the image of $\mbM$, we define the \emph{$\mbM$-span} of the $X_j$, $\mbG_\mbM(\{X_j\})$ to be $$\add\{\mbM(F)X_j|F\in\cC(\tti,\ttk)\text{ for some }\tti,\ttk\in\cC, j\in J\},$$ where $\add S$ is the \emph{additive closure} of the set $S$, defined as the smallest full subcategory of $\ceM$ containing $S$ and closed under direct sums and direct summands.\end{defn}

Defining $$\mbG_\mbM(\{X_j\})(\ttk):=\add\{\mbM(F)X_j|F\in\cC(\tti,\ttk)\text{ for some }\tti\in\cC, j\in J\},$$ it is immediate from the definition that $\mbG_\mbM(\{X_j\})$ is a sub-2-representation of $\mbM$.

\begin{defn} Let $\mbM$ be a finitary 2-representation of $\cC$. We say that $\mbM$ is \emph{transitive} if for any $\tti\in\cC$ and any non-zero $X\in\mbM(\tti)$, $\mbG_\mbM(\{X\})$ is equivalent to $\mbM$.\end{defn}

To define simple transitive 2-representations we need the following generalisation of \cite{mazorchuk2015transitive}:

\begin{lem}\label{locfinmid} Let $\mbM$ be a transitive 2-representation of $\cC$. There exists a unique maximal ideal $\cI$ of $\mbM$ such that $\cI$ does not contain any identity morphisms apart from the zero object.
	\begin{proof} The proof is mutatis mutandis given by the proofs of \cite[Lemma 3, 4]{mazorchuk2015transitive}.\end{proof}\end{lem}

\begin{defn} A transitive 2-representation $\mbM$ is \emph{simple transitive} if the maximal ideal of $\mbM$ given in \autoref{locfinmid} is the zero ideal.\end{defn}

Let $\cC$ now be a locally finitary 2-category with a $\cJ$-cell $\ceJ$ and a $\cL$-cell $\ceL\subseteq\ceJ$. We let $\tti=\tti_{\ceL}$ be the object of $\cC$ which is the domain of every $F\in\ceL$. For each $\ttj\in \cC$, we define a full subcategory $\mbN_\ceL(\ttj)$ of $\mbP_\tti(\ttj)$ by $$\mbN_\ceL(\ttj)=\add\{FX|F\in\cC(-,\ttj), X\in\ceL\}.$$ $\mbN_\ceL$ thus defines a map from the object set of $\cC$ into $\LA_\bbk$. To make this a 2-representation, of $\cC$, we take a 1-morphism $F$ to the functor defined by left composition with $F$, and a 2-morphism $\alpha$ is taken to the natural transformation defined by left horizontal composition with $\alpha$.
\par
By a similar proof to that of \autoref{locfinmid} (and hence to that of \cite[Lemma 3]{mazorchuk2015transitive}), $\mbN_\ceL$ is a transitive 2-representation, and thus has a unique maximal ideal not containing any identity morphisms for non-zero objects, and we can take its simple transitive quotient:

\begin{defn}\label{cell2repdefn} The simple transitive quotient of $\mbN_\ceL$ is the \emph{cell 2-representation} of $\cC$ corresponding to the $\cL$-cell $\ceL$. We generally denote this quotient as $\mbC_\ceL$. \end{defn}

We now define the apex of a finitary 2-representation, following \cite{chan2016diagrams}.

\begin{defn}\label{ApexDef} Let $\mbM$ be a finitary 2-representation of a locally finitary 2-category $\cC$. Let $S_\cJ(\cC)$ denote the poset of $\cJ$-cells of $\cC$, ordered by $\leq_\cJ$. If there exists a unique maximal $\ceJ\in S_\cJ(\cC)$ that is not annihilated by $\mbM$, then $\ceJ$ is called the \emph{apex} of $\mbM$.\end{defn}

\begin{defn} Let $\cC$ be a locally finitary 2-category and let $\ceJ$ be a $\cJ$-cell in $\cC$. We say that $\ceJ$ is \emph{non-trivial} if it contains some non-identity 1-morphism. Otherwise, $\ceJ$ is \emph{trivial}. We similarly define trivial and non-trivial $\cL$- and $\cR$-cells, and define a \emph{trivial} cell 2-representation to be a cell 2-representation associated to a trivial $\cL$-cell.\end{defn}

\section{Coalgebra And Comodule 1-Morphisms For Locally Finitary 2-Categories}\label{MMMTGen} We now generalise useful results from \cite{mackaay2016simple}. To start, we utilise the abelianisation constructions given there in Sections 3.2, 3.3 and 3.4, with the latter two generalising without issue to this setup. We recall the 1-categorical version for later reference.

\begin{defn}\label{fincelldef} Given an additive category $\ceC$, the \emph{injective fan Freyd abelianisation} $\ul{\ceC}$ of $\ceC$ is an additive category that is defined as follows: \begin{itemize}
		
		\item Objects of $\ul{\ceC}$ are equivalence classes of tuples of the form $(X,k,Y_i,f_i)_{i\in\bbZ^+}$ where $X$ and the $Y_i$ are objects of $\ceC$, the $f_i: X\to Y_i$ are morphisms of $\ceC$, and $k$ is a non-negative integer such that $Y_i=0$ for $i>k$. Two tuples are equivalent if they only differ in the value of $k$.
		\item A morphism from $(X,k,Y_i,f_i)$ to $(X',k',Y_i',f'_i)$ is an equivalence class of tuples $(g,h_{ij})_{i,j\in\bbZ^+}$ where $g:X\to X'$ and $h_{ij}:Y_i\to Y'_j$ are morphisms of $\ceC$ such that $f'_ig=\sum_j h_{ji}f_j$ for each $i$, modulo the $(g,h_{ij})$ such that there exist $q_i:Y_i\to X'$ with $\sum_i q_if_i=g$.
		\item The identity morphisms are of the form $(\id_X,\delta_{ij}\id_{Y_i})$ and composition is given by $(g',h'_{ij})\circ(g,h_{ij})=(g'g,\sum_kh'_{kj} h_{ik})$.\end{itemize}
	
	The \emph{projective fan Freyd abelianisation} is defined dually.\end{defn}

To provide a more intuitive picture of the above definition, the objects of $\ul{\ceC}$ are `fans' of objects of $\ceC$, specifically of the form \[\xymatrix{ & Y_1\\ X \ar[r]|{\,f_2\,}\ar[ru]^{f_1}\ar[rd]_{f_i} & Y_2 \\ & \vdots}\] where the $Y_i$ are all equal to zero for sufficiently large $i$. Morphisms between two fans consist of a single morphism between the `heads' of the fans and a matrix of morphisms between the non-zero `leaves' of the fans, subject to the above equivalence relation.

\subsection{The Coalgebra Construction} Let $\cC$ be a locally finitary 2-category and let $\mbM$ be a transitive 2-representation of $\cC$. Throughout the rest of this paper, we will be referring to (internal) coalgebra and comodule 1-morphisms. These are the natural generalisation of e.g. coalgebras in monoidal categories - see \cite[Section 6.4]{johnson20212} for a more detailed definition of coalgebra 1-morphisms, under the name `comonads'. We use the terms `coalgebra' and `comodule' over `comonad' and `coalgebra' to retain continuity with other finitary 2-representation papers.\par

Given any $S\in\mbM(\tti)$, we define the evaluation functor $\ev_S:\coprod_{\ttj\in\cC}\ul{\cC}(\tti,\ttj)\to\ul{\ceM}$ taking $F$ to $FS$ and $\alpha:F\to G$ to $\alpha_S: FS\to GS$. Since $\ev_S$ maps each $F\in\ul{\cC}(\tti,\ttj)$ to an object in $\ul{\mbM}(\ttj)$, this functor has a left adjoint if and only if each of the component evaluation functors $\ev_{S,\ttj}:\ul{\cC}(\tti,\ttj)\to\ul{\mbM}(\ttj)$ does so. But the latter case is the finitary case where such adjoints exist by e.g. \cite[Section 1.3]{chan2016diagrams}, and  we construct a left adjoint $[S,-]:\coprod\limits_{\ttj\in\cC}\ul{\mbM}(\ttj)\to\coprod\limits_{\ttj\in\cC}\ul{\cC}(\tti,\ttj)$. We give the following generalisation of \cite[Lemma 4.3]{mackaay2016simple}:

\begin{lem}\label{LFMMMTL5} With the above notation, $[S,S]$ has the structure of a coalgebra 1-morphism in $\ul{\cC}(\tti,\tti)$.
	\begin{proof} We mirror the proof of \cite[Lemma 4.3]{mackaay2016simple}. The image of $\id_{[S,S]}$ under the adjunction isomorphism $\Hom_{\ul{\cC}}([S,S],[S,S])\to\Hom_{\ul{\mbM}}(S,[S,S]S)$ gives a non-zero morphism $\coev_S:S\to[S,S]S$. This gives the composition\newline \begin{center}
		 \xymatrix{ S \ar[rr]^-{\coev_S} & & [S,S]S \ar[rr]^-{[S,S]\coev_S} & & [S,S][S,S]S}\end{center}
		from $S$ to $[S,S][S,S]S$. But then as $$\Hom_{\ul{\mbM}}(S,[S,S][S,S]S)\cong\Hom_{\ul{\cC}}([S,S],[S,S][S,S]),$$ again by the adjunction isomorphism, this gives a non-zero comultiplication 2-morphism $[S,S]\to[S,S][S,S]$.\par
		
		For the counit morphism we again use the adjunction isomorphism to note that $\Hom_{\ul{\mbM}}(S,S)=\Hom_{\ul{\mbM}}(S,\bbon_\tti S)\cong\Hom_{\ul{\cC}}([S,S],\bbon_\tti)$, and thus choose the (non-zero) image of $\id_S$ under this isomorphism. We denote the comultiplication 2-morphism by $\Delta_S$ and the counit 2-morphism by $\epsilon_S$. It is straightforward to check that the coalgebra axioms hold.\end{proof}\end{lem}
	
	We denote the category of comodule 1-morphisms of $[S,S]$ by $\comod_{\ul{\cC}}([S,S])$ and its subcategory of injective comodules by $\inj_{\ul{\cC}}([S,S])$. These can be considered instead as 2-representations of $\cC$ and $\ul{\cC}$ in a natural fashion; when considering them as such we denote them as $\bcomod_{\ul{\cC}}([S,S])$ and $\binj_{\ul{\cC}}([S,S])$ respectively.\par
	
	For the rest of the section we will assume that $\cC$ is a locally weakly fiat 2-category.\par
	
	Given any $T\in\ceM$, $[S,T]$ can be considered as a right $[S,S]$-comodule 1-morphism in a canonical fashion with coaction 2-morphism $\rho_{[S,T]}:[S,T]\to[S,T][S,S]$, or $\rho_T$ when there is no confusion. We define a functor $\Theta:\cM\to\comod_{\ul{\cC}}([S,S])$ by setting $$\Theta(T)=([S,T],\rho_T)\text{; }\Theta(f)=[S,f].$$
	
	\begin{lem}\label{MMMTL44} The functor $\Theta$ (weakly) commutes with the action of $\cC$ and defines a morphism of 2-representations.
		\begin{proof} This is a generalisation of \cite[Lemma 4.4]{mackaay2016simple}, and the proof given there shows that $[S,XT]\cong X[S,T]$ in $\ul{\cC}$ for any 1-morphism $X$ in $\cC$ by referring to only hom-categories between three objects of $\cC$. The proof is therefore entirely local and generalises immediately.\end{proof}\end{lem}
	
	We also present generalisations of \cite[Lemmas 4.5, 4.6]{mackaay2016simple}:
	
	\begin{lem}\label{MMMTL45} For any 1-morphism $X$ in $\cC$ and any $C\in\comod_{\ul{\cC}}([S,S])$ there is an isomorphism $\Hom_{\comod_{\ul{\cC}}([S,S])}(C,X[S,S])\cong\Hom_{\ul{\cC}}(C,X)$.
		\begin{proof} The proof given in \cite{mackaay2016simple} is an entirely local proof and generalises immediately.\end{proof}\end{lem}
	
	\begin{lem}\label{MMMTL46} $\Theta$ factors over the inclusion $\inj_{\ul{\cC}}([S,S])\hookrightarrow\comod_{\ul{\cC}}([S,S])$.
		\begin{proof} We mirror the proof given for \cite[Lemma 4.6]{mackaay2016simple} with some extra detail to clarify it for our situation. We first consider the case where $T=XS$ for some 1-morphism $X\in\cC(\tti,\ttj)$. By \autoref{MMMTL44} we have that $[S,T]=[S,XS]\cong X[S,S]$. By the definition of a comodule, $[S,S]$ is injective in $\comod_{\ul{\cC}}([S,S])$. We claim that, because $\cC$ is locally weakly fiat, $X[S,S]$ is also injective. For the existence of internal adjunctions in $\cC$ gives that $$\Hom_{\comod_{\ul{\cC}}([S,S])}(-,X[S,S])\cong\Hom_{\comod_{\ul{\cC}}([S,S])}({}^*X-,[S,S])$$ and as the latter is exact by the injectivity of $[S,S]$ and by $\tensor[^*]{X}{}$ having both left and right adjoints, so is the former. But since $\mbM$ is transitive, any $T$ is isomorphic to a direct summand of some $XS$. The result follows.\end{proof}\end{lem}
	
	We can now give the generalisation of \cite[Theorem 4.7]{mackaay2016simple}:
	
	\begin{thm}\label{MMMTL47} Take $\cC$ to be a locally weakly fiat 2-category, $\mbM$ to be a transitive 2-representation of $\cC$, and $S\in\mbM(\tti)$ to be non-zero. Letting $\Theta$ be the functor defined above, $\Theta$ induces a 2-representation equivalence between $\ul{\mbM}$ and $\bcomod_{\ul{\cC}}([S,S])$. This restricts to an equivalence between $\mbM$ and $\binj_{\ul{\cC}}([S,S])$.
		\begin{proof} The proof generalises with only minor notes from the proof for \cite[Theorem 4.7]{mackaay2016simple}. The references \cite[Lemmas 4.4, 4.5, 4.6]{mackaay2016simple} found therein are here replaced with Lemmas \ref{MMMTL44}, \ref{MMMTL45} and \ref{MMMTL46} respectively. We finally note that our construction above still has $[S,S]$ as an injective cogenerator of $\comod_{\ul{\cC}}([S,S])$, allowing the argument given in the aforementioned proof to generalise.	\end{proof}\end{thm}

We note that \cite[Corollary 4.10]{mackaay2016simple} is a local corollary and thus also generalises immediately:

\begin{cor}\label{MMMTC410} For $\tti\in\cC$, consider the endomorphism 2-category $\cC_\tti$ of $\tti$ in $\cC$. There is a bijection between the equivalence classes of simple transitive 2-representations on $\cC_\tti$ and simple transitive 2-representations in $\cC$ which are non-zero at $\tti$.\end{cor}

\subsection{The 2-category $\cC_A$}\label{cCA}

Let $A=\{A_\tti\}_{\tti\in I}$ be a countable collection of basic self-injective connected finite dimensional $\bbk$-algebras. We define the locally finitary 2-category $\cC_A$ to have as objects the $\tti\in I$, which we associate with (small categories equivalent to) $A_\tti$-$\modd$. The 1-morphisms of $\cC_A$ are direct sums of functors from $A_\tti$-$\modd$ to $A_\ttj$-$\modd$ whose direct summands are isomorphic to the identity functor or to tensoring with projective $(A_\ttj$-$A_\tti)$-bimodules. The 2-morphisms of $\cC_A$ are natural transformations of functors. As a small example, if $A=\bbk$ then $\cC_\bbk$ has one object $*_\bbk$ and the only 1-morphisms are functors isomorphic to tensoring with direct sums of $\bbk\otimes_\bbk\bbk\cong\bbk$, i.e. isomorphic to direct sums of $\bbon_{*_\bbk}$.

$\cC_A$ is a locally finitary 2-category by a similar argument to that found in \cite[Section 7.3]{mazorchuk2011cell}, and indeed is actually a locally weakly fiat 2-category (again, by a similar argument to that found in \cite[Section 5.1]{mazorchuk2015transitive}).\par

\begin{defn}\label{CAXDef} Let $Z_\tti$ denote the centre of $A_\tti$. Identifying $Z_\tti$ with $\End_{\cC_A}(\bbon_{\tti})$, we denote by $Z'_\tti$ the subalgebra of $Z_\tti$ that is generated by $\id_{\bbon_\tti}$ and all elements of $Z_\tti$ which factor through 1-morphisms equivalent to tensoring with projective $(A_\tti$-$A_\tti)$-bimodules. We now choose subalgebras $X_\tti$ of $Z_\tti$ containing $Z'_\tti$, and let $X=\{X_\tti|\tti\in I\}$. We define a sub-2-category $\cC_{A,X}$ of $\cC_A$ which has the same objects and 1-morphisms, and the same 2-morphisms except that $\End_{\cC_{A,X}}(\bbon_\tti):=X_\tti$.\end{defn}

Up to isomorphism, the non-identity indecomposable 1-morphisms of $\cC_A$ or $\cC_{A,X}$ are of the form $$A_\ttj e_{\ttj l}\otimes_\bbk e_{\tti k}A_\tti\otimes_{A_\tti}-$$ where $\tti,\ttj\in I$ and $e_{\tti 1},\dots,e_{\tti n_\tti}$ (respectively $e_{\ttj 1},\dots,e_{\ttj n_{\ttj}}$) are a complete set of primitive orthogonal idempotents of $A_\tti$ (respectively $A_\ttj$). For notational compactness, we define $$F^{\tti k}_{\ttj l}:=A_\ttj e_{\ttj l}\otimes_\bbk e_{\tti k}A_\tti\otimes_{A_\tti}-$$ in $\cC_A$ or $\cC_{A,X}$. We also denote $A^{\tti k}_{\ttj l}:=A_\ttj e_{\ttj l}\otimes_\bbk e_{\tti k}A_\tti$ for the corresponding bimodule.\par

We present the generalisation of \cite[Lemma 12]{mazorchuk2016endomorphisms}:

\begin{lem}\label{CAXlwf} $\cC_{A,X}$ is well-defined and locally weakly fiat.
	\begin{proof} Since the only difference between $\cC_{A,X}$ and $\cC_A$ is in the endomorphism spaces of the $\bbon_\tti$, the proof given in \cite{mazorchuk2016endomorphisms} is entirely local and generalises without issue to this setup.\end{proof}\end{lem}

Excluding the trivial $\cJ$-cells, there is a single $\cJ$-cell of $\cC_{A,X}$ consisting of all the $A_\ttj e_{\ttj l}\otimes_\bbk e_{\tti k}A_\tti$. The left cells are of the form $$\ceL_{\tti k}=\{F^{\tti k}_{\ttj l}|\ttj\in I,l=1,\dots, n_\ttj\}$$ and the right cells are of the form $$\ceR_{\ttj l}=\{F^{\tti k}_{\ttj l}|\tti\in I,k=1,\dots, n_\tti\}.$$\par

We examine the structure of cell 2-representations of $\cC_{A,X}$ more explicitly. Let $\ceL_{\tti k}$ be a left cell as defined previously. The corresponding 2-representation $\mbN_{\tti k}:=\mbN_{\ceL_{\tti k}}$ has indecomposable objects $F^{\tti k}_{\ttj l}\in\mbN_{\tti k}(\ttj)$ for $l=1,\dots n_\ttj$. A bimodule homomorphism $\sigma:A^{\tti k}_{\ttj l}\to A^{\tti k}_{\ttj m}$ is defined by its image on $e_{\ttj l}\otimes e_{\tti k}$, and is thus a $\bbk$-linear combination of homomorphisms of the form $\varphi_{a,b}:A^{\tti k}_{\ttj l}\to A^{\tti k}_{\ttj m}$, where $\varphi_{a,b}(e_{\ttj l}\otimes e_{\tti k})=a\otimes b$ for $a\in e_{\ttj l}A_\ttj e_{\ttj m}$ and $b\in e_{\tti k} A_\tti e_{\tti k}$. We also note that $\id_{F^{\tti m}_{\ttj l}}=\varphi_{e_{\ttj l}, e_{\tti m}}$ (abusing notation to let $\varphi_{a,b}$ refer both to the bimodule homomorphism and the corresponding 2-morphism in $\cC_{A,X}$).

\begin{psn}\label{CAXCellStr} Let $\ceI$ be the unique maximal ideal of $\mbN_{\tti k}$ not containing any identity morphisms for non-zero objects (so that $\mbN_{\tti k}/\ceI$ is a cell 2-representation). Then its components $\ceI(\ttj)\subseteq\mbN_{\tti k}(\ttj)$ are matrices of $\bbk$-linear combinations of morphisms of the form $\varphi_{a,b}$ with $a\in A_\ttj$ and $b\in R:=\rad e_{\tti k}A_\tti e_{\tti k}$.
	\begin{proof} Since the $\mbN_{\tti k}(\ttj)$ are additive categories, by composing elements of $\coprod_{\ttj} \ceI(\ttj)$ with biproduct injections and projections it is sufficient to consider the elements of $\coprod_{\ttj} \ceI(\ttj)$ that are morphisms between indecomposable objects. We refer to this process as an injection-projection argument. First, given elements $a,\alpha,\gamma\in A_{\ttj}$, $\beta,\delta\in A_{\tti}$ and $b,b'\in \rad e_{\tti k}A_\tti e_{\tti k}$, $\varphi_{\alpha,\beta}\varphi_{a,b}\varphi_{\gamma,\delta}=\varphi_{\gamma a\alpha, \beta b\delta}$ and $\varphi_{a,b}+\varphi_{a,b'}=\varphi_{a,b+b'}$. Since $R$ is a two-sided ideal of $e_{\tti k}A_\tti e_{\tti k}$, this implies that $\ceI$ is indeed an ideal of $\ceN_{\tti k}$. Further, if $\id_{F^{\tti k}_{\ttj l}}\in\ceI(\ttj)$ then $e_{\tti k}\in R$, which is a contradiction as $R$ is a proper ideal of $e_{\tti k}A_\tti e_{\tti k}$. Hence $\ceI$ does not contain $\id_X$ for non-zero $X\in\ceN_{\tti k}$.\par
		
		To show that $\ceI$ is $\cC_{A,X}$-stable, it is again sufficient by an injection-projection argument to consider the action of indecomposable 1-morphisms of $\cC_{A,X}$. Choose $F^{\ttj m}_{\ttl n}\in\cC_{A,X}(\ttj, \ttl)$ and $\varphi_{a,b}:F^{\tti k}_{\ttj s}\to F^{\tti k}_{\ttj t}\in\ceI(\ttj)$. Then $$\mbN_{\tti k}(F^{\ttj m}_{\ttl n})(\varphi_{a,b})=\varphi_{e_{\ttl n},e_{\ttj m}}\otimes \varphi_{a,b}:A^{\ttj m}_{\ttl n}\otimes_{A_\ttj} A^{\tti k}_{\ttj s}\to A^{\ttj m}_{\ttl n}\otimes_{A_\ttj} A^{\tti k}_{\ttj t}.$$ Under the isomorphism $$A^{\ttj m}_{\ttl n}\otimes_{A_\ttj} A^{\tti k}_{\ttj t}\cong A_{\ttl n}\otimes_\bbk e_{\ttj m}A_\ttj e_{\ttj t}\otimes_{\bbk} e_{\tti k}A_\tti\cong (A^{\tti k}_{\ttl n})^{\oplus\dim e_{\ttj m}A_\ttj e_{\ttj t}},$$ the element $$(\varphi_{e_{\ttl n},e_{\ttj m}}\otimes \varphi_{a,b})(e_{\ttl n}\otimes e_{\ttj m}\otimes e_{\ttj s}\otimes e_{\tti k})=e_{\ttl n}\otimes e_{\ttj m}\otimes a\otimes b$$ maps first to $e_{\ttl n}\otimes e_{\ttj m}a\otimes b$ and then to $$\bigoplus_{x=1}^{\dim e_{\ttj m}A_\ttj e_{\ttj t}} v_x e_{\ttl n}\otimes b,$$ for some $v_x\in\bbk$. Since $b\in R$, this implies that $\mbN_{\tti k}(F^{\ttj m}_{\ttl n})(\varphi_{a,b})\in\ceI(\ttl)$, giving $\cC_{A,X}$-stability.\par
		
		It remains to show that $\ceI$ is the unique maximal such ideal. If it is maximal, then \autoref{locfinmid}, it is immediate that it is unique. Thus assume for contradiction that there exists some other ideal $\ceK\supset\ceI$ such that $\ceK$ does not contain $\id_F$ for any non-zero $F$. Choose some $\sigma\in\ceK\setminus\ceI$. By injection-projection arguments we may assume that $\sigma$ is a morphism from $F^{\tti k}_{\ttj l}$ and $F^{\tti k}_{\ttj m}$. Thus $\sigma=\sum_{v=1}^t \varphi_{a_v,b_v}$ for some $t$ with $a_v\in e_{\ttj l}A_\ttj e_{\ttj m}$ and $b_v\in e_{\tti k} A_\tti e_{\tti k}$.\par
		
		 If $b_v\in R$ for some $v$, then by definition $\varphi_{a_v,b_v}\in\ceI\subset\ceK$, and thus without loss of generality $b_v\notin R$ for all $v$. But then by \cite[Lemma I.4.6]{assem2006elements}, $e_{\tti k}-b_v\in R$ for all $v$, and hence $$\sigma+\sum_{v=1}^t \varphi_{a_v,e_{\tti k}-b_v}=\sum_{v=1}^t \varphi_{a_v,e_{\tti k}}=\varphi_{\sum_v a_v, e_{\tti k}}\in\ceK.$$ But since $\ceK$ is $\cC_{A,X}$-stable, by tensoring  $\varphi_{e_{\ttj m}, e_{\tti k}}$ with $\varphi_{\sum a_v, e_{\tti k}}$ similarly to above and composing with injection and projection morphisms, we therefore derive that $z\varphi_{e_{\ttj m}, e_{\tti k}}\in\ceK$ for $z\in\bbk$. But this implies $\id_{F^{\tti k}_{\ttj m}}\in\ceK$, a contradiction. Thus $\ceI$ is indeed maximal, and the result follows.\end{proof}\end{psn}
	 
We can use \autoref{CAXlwf} to provide a simple proof of the generalisation of \cite[Theorem 15]{mazorchuk2015transitive}. However, we will first give a simple proof of the connected version of that Theorem:

\begin{psn}\label{smallcell} Let $A$ be a basic self-injective connected finite dimensional $\bbk$-algebra. Then for every simple transitive 2-representation of $\cC_A$ there is an equivalent cell 2-representation.
	\begin{proof} We consider a larger weakly fiat 2-category $\cC_{A\times\bbk}$. This category has two objects $*$ and $*_\bbk$. We identify the former with a small category $\ceA$ equivalent to $A$-$\modd$ and the latter with a small category equivalent to $\bbk$-$\modd$. The category $\cC_{A\times\bbk}(*,*)$ is taken to be $\cC_A(*,*)$ and the category $\cC_{A\times\bbk}(*_\bbk,*_\bbk)$ is taken to be $\cC_{\bbk}(*_\bbk,*_\bbk)$. The 1-morphisms between $*$ and $*_\bbk$ (respectively $*_\bbk$ and $*$) are direct summands of direct sums of functors isomorphic to tensoring with projective $(\bbk$-$A)$-bimodules (respectively projective $(A$-$\bbk)$-bimodules). The 2-morphisms are natural transformations.\par
		
		The endomorphism 2-category of $*$ is equivalent to $\cC_A$. Further, if we have an indecomposable 1-morphism $Ae_i\otimes_\bbk e_jA\otimes_A-$ in $\cC_{A\times\bbk}(*,*)$, then $Ae_i$ is an $(A$-$\bbk)$-bimodule and $e_jA$ is a $(\bbk$-$A)$-bimodule, and hence this factors over $*_\bbk$.\par
		
		We denote by $\cC_\bbk$ the endomorphism 2-category of $*_\bbk$, and claim that any simple transitive 2-representation of it is equivalent to the cell 2-representation. Let $\mbM$ be a simple transitive 2-representation of $\cC_\bbk$. As the 1-morphisms in $\cC_\bbk$ are all of the form $\bbon_{*_\bbk}^{\oplus m}$ for some $m\in\bbZ^+_0$, let $N\in\mbM(*_\bbk)$ be indecomposable. Then as $\mbM$ is transitive and $\mbM(*_\bbk)$ is idempotent complete, $M\cong\id^{\oplus n}(N)\cong N^{\oplus n}$ for any $M\in\mbM(*_\bbk)$ and for some $n\in\bbZ^+_0$. It follows from \autoref{CAXCellStr} that for the left cell $\ceL_\bbk=\{\bbon_{*_\bbk}\}$ of $\cC_{\bbk}$, $\mbP_{*_\bbk}=\mbN_{\ceL_\bbk}=\mbC_{\ceL_\bbk}$ as $\rad\bbk=0$. Let $\Phi:\mbP_{*_\bbk}\to\mbM$ be a morphism of 2-representations defined on objects by $\Phi(F)=\mbM(F)(N)$ and on morphisms by $\Phi(f)=\mbM(f)_N$. We can abuse notation and equate $\Phi$ with $\Phi_{*_\bbk}:\mbP_{*_\bbk}(*_\bbk)\to\mbM(*_\bbk)$. It is immediate from the prior discussion that $\Phi$ is essentially surjective on objects and faithful.\par
		
		To show that $\Phi$ is full, let $\ceK$ be the $\cC$-stable ideal of $\mbM(*_\bbk)$ that is generated by $\rad\End_\mbM(N)$. Since $N$ is indecomposable, by a categorical variant of \cite[Corollary I.4.8]{assem2006elements} $\End_{\mbM}(N)$ is local and $\rad\End_{\mbM}(N)$ is the unique maximal ideal. Assume for contradiction that $\id_M\in\ceK$ for some $M\in\mbM(*_\bbk)$. Then by standard injection-projection arguments $\id_N\in\ceK$. But since any morphism $f:N^{\oplus m}\to N^{\oplus n}$ is an $m\times n$ matrix of elements of $\End_{\mbM}(N)$, this implies that $\id_N=\sum_{i=1}^n f_i k_i g_i$ for some morphisms $f_i,g_i\in\End_\mbM(N)$ for all $i$ and morphisms $k_i\in\rad\End_{\mbM}(N)$ for all $i$, i.e. that $\id_N\in\rad\End_\mbM(N)$, a contradiction. Hence $\ceK$ does not contain $\id_M$ for any $M\in\mbM(*_\bbk)$. But since $\mbM$ is simple transitive by assumption, this implies that $\ceK=0$ and thus that $\End_\mbM(N)\cong\bbk$. It follows immediately that $\Phi$ is full, and thus an equivalence of 2-representations as we wished to show.\par
		
		Returning to the main aim of the proof, if $A=\bbk$ then we are done by the above work. Hence assume that $A\neq\bbk$. Using the previous paragraph and \autoref{MMMTC410}, we know that there is a unique equivalence class of simple transitive 2-representations of $\cC_{A\times\bbk}$ that is non-zero on $*_\bbk$. We now claim that if a 2-representation $\mbN$ of $\cC_{A\times\bbk}$ is non-zero on $*$, then it is either non-zero on $*_\bbk$ or equivalent to the trivial cell 2-representation on $\cC_A$. \par
		
		If $Ae_i\otimes_\bbk e_jA$ acts in a non-zero fashion on $\mbN(*)$ for some $i$ or $j$, then as it factors through $\mbN(*_\bbk)$, we must have that $\mbN(*_\bbk)$ is non-zero. Therefore assume that $Ae_i\otimes_\bbk e_jA$ acts as the zero functor for every $i$ and $j$. Then the only 1-morphisms in $\cC_A(*,*)$ that act non-trivially on $\mbN(*)$ are direct sums of $\bbon_*$. In particular, if $N\in\mbN(*)$ is indecomposable, then for any $M\in\mbN(*)$, $M\cong N^{\oplus n}$ for some $n$. But this is equivalent to the cell 2-representation for the trivial $\cL$-cell by a similar argument to above.\par
		
		It follows by \autoref{MMMTC410} that there is only one equivalence class of simple transitive 2-representations on $\cC_A$ that is not equivalent to the identity cell 2-representation, and as we know that $\cC_A$ has a cell 2-representation for the maximal $\cJ$-cell (which by assumption is distinct from $\{[\bbon_*]\}$) the result follows.\end{proof}\end{psn}

This leads to the main theorem of this section:

\begin{thm}\label{MM5ByMMMT} Let $A=\{A_\tti|\tti\in I\}$ be a countable collection of basic self-injective connected finite dimensional $\bbk$-algebras. Any non-zero simple transitive 2-representation of $\cC_A$ is equivalent to a cell 2-representation.
	\begin{proof} If $A=\{\bbk\}$ then we are done by the proof of \autoref{smallcell}. Assume that $A\neq\{\bbk\}$. Let $\mbM$ be a simple transitive 2-representation of $\cC_A$. Assume that there is some $\ttj$ such that $\mbM(\ttj)=0$ and let $\tti$ be such that $\mbM(\tti)\neq 0$. Then $$A_\tti e_{\tti k}\otimes_\bbk e_{\ttj l}A_\ttj\otimes_{A_\ttj} A_\ttj e_{\ttj m}\otimes_\bbk e_{\tti n} A_\tti\otimes_{A_\tti} -$$ is the zero map for any primitive idempotents $e_{\tti k}$ and any $e_{\tti n}$. But in particular $$A_\tti e_{\tti k}\otimes_\bbk e_{\ttj m}A_\ttj\otimes_{A_\ttj} A_\ttj e_{\ttj m}\otimes_\bbk e_{\tti n} A_\tti \cong (A_\tti e_{\tti k}\otimes_\bbk e_{\tti n} A_\tti)^{\oplus\dim e_{\ttj m}A_\ttj e_{\ttj m}},$$ and this implies that $A_\tti e_{\tti j}\otimes_\bbk e_{\tti m}A_\tti\otimes_{A_\tti}-$ is the zero map on $\mbM(\tti)$ for any $j$ and $m$. But by a similar argument to \autoref{smallcell}, this means $\mbM$ is equivalent to a cell 2-representation for an identity cell. Thus if $\mbM$ is not equivalent to a trivial cell 2-representation, it must follow that $\mbM$ is non-zero on every $\tti$.\par
		
		Assume that $\mbM(\ttj)\neq 0$ for all $\ttj\in\cC_A$, and choose some $\tti\in\cC_A$. By \autoref{smallcell}, every simple transitive 2-representation of $\cC_{A_\tti}$ is equivalent to a cell 2-representation, and in particular there is only one equivalence class of simple transitive 2-representations that is not equivalent to the trivial cell 2-representation. But as every simple transitive 2-representation of $\cC_A$ not equivalent to a trivial cell 2-representation is non-zero when it restricts down to $\cC_{A_\tti}$ it follows from \autoref{MMMTC410} that there is only a single equivalence class of simple transitive 2-representations of $\cC_A$ not equivalent to a trivial cell 2-representation. Since $\cC_A$ has a cell 2-representation for the maximal $\cJ$-cell, the result follows. \end{proof}\end{thm}

The above arguments in \autoref{smallcell} and \autoref{MM5ByMMMT} do not depend on the endomorphisms of $\bbon_\tti$ for any object $\tti$ of $\cC$. \autoref{MM5ByMMMT} therefore generalises without notable change to the following:

\begin{cor}\label{MM5ByMMMTAX} Taking $A=\{A_\tti|\tti\in I\}$ as before, let $X=\{X_\tti|\tti\in I\}$ be a collection of subalgebras $X_\tti\subseteq A_\tti$ as in \autoref{CAXDef}. Then any non-zero simple transitive 2-representation of $\cC_{A,X}$ is equivalent to a cell 2-representation.\end{cor}

As a supplementary note, the method of the proof of \autoref{smallcell} allows us to give the following result:

\begin{cor}\label{celldef} All non-trivial cell 2-representations of $\cC_A$ (respectively $\cC_{A,X}$) are equivalent.\end{cor}

There are times where we wish to consider the situation where the $A_\tti$ are not necessarily basic. Given a non-basic algebra $A_{\tti}$, let $\{e_{\tti 1}^1,\dots,e_{\tti 1}^{n_1},e_{\tti 2}^1,\dots e_{\tti m}^{n_m}\}$ be a complete set of idempotents of $A_\tti$ such that $A_\tti e_{\tti j}^k\cong A_\tti e_{\tti l}^p$ if and only if $j=l$. Letting $e^b=e_{\tti 1}^1+e_{\tti 2}^1+\dots+e_{\tti m}^1$, we define the basic algebra $A_\tti^b$ associated to $A_\tti$ as $A_\tti^b=e^bA_\tti e^b$ (see \cite[Section I.6]{assem2006elements} for further discussion). Note that if $A_\tti$ is basic, then $A_\tti^b=A_\tti$.\par

Given a countable collection $A=\{A_\tti\}_{\tti\in I}$ of self-injective connected finite dimensional $\bbk$-algebras which are not necessarily basic, we define $A^b=\{A_\tti^b\}_{\tti\in I}$ and consequently define $\cC_A:=\cC_{A^b}$, the latter as defined at the start of this section.

\section{Locally Weakly Fiat 2-Categories And Their Simple Transitive 2-Representations}\label{MMXGen}

We move on to extending results from the Mazorchuk--Miemietz series of papers to the locally finitary case, with the eventual aim of generalising \cite[Theorem 18]{mazorchuk2015transitive} to show that any simple transitive 2-representation of a strongly regular locally weakly fiat 2-category is equivalent to a cell 2-representation. We note that here we are using the projective fan Freyd abelianisation rather than the injective equivalent as we did above. Many of the proofs of this section are straightforward generalisations of proofs found elsewhere in the literature; we will only provide proofs of results that are novel or need a non-trivial amount of work to generalise; for the others, we will merely provide the reference for the original result.\par

For the entirety of this section, we assume that $\cC$ is a locally weakly fiat 2-category.

\subsection{General Properties Of The Abelianisation}

We start by giving some general properties of the action of 1-morphisms of $\cC$ on its $\tti$th abelian principal 2-representation $\ol{\mbP_\tti}$. The isomorphism classes of indecomposable projectives and simples are indexed by the isomorphism classes of indecomposable 1-morphisms in $\cC$ and we denote them as $P_F$ and $L_F$ respectively.

\begin{lem}\label{MM1P12} Let $F,G$ be indecomposable 1-morphisms of $\cC$. Then $FL_G\neq 0$ if and only if $F\leq_\cL G^*$.
	\begin{proof} This is a generalisation of \cite[Lemma 12]{mazorchuk2011cell}. Though the proof is similar, we note that there are alterations needed due to how adjunctions work in the (locally) weakly fiat case.\par
		
		Assume that $G\in\cC(\tti,\ttj)$ and $F\in\cC(\ttj,\ttk)$. We note that $FL_G\neq 0$ if and only if there exists some indecomposable $H\in\cC(\tti,\ttk)$ with $\Hom_{\ol{\cC}(\tti,\ttk)}(P_H, FL_G)\neq 0$. We have the adjunction $$0\neq\Hom_{\ol{\cC(\tti,\ttk)}}(P_H,FL_G)\cong\Hom_{\ol{\cC(\tti,\ttj)}}(\tensor[^*]{F}{}\circ HP_{\bbon_\tti},L_G),$$ and as $\tensor[^*]{F}{}\circ HP_{\bbon_\tti}$ is projective and $L_G$ is simple, the above inequality is equivalent to saying that $P_G=GP_{\bbon_\tti}$ is a direct summand of ${}^*F\circ HP_{\bbon_\tti}$, i.e. $G$ is a direct summand of ${}^*F\circ H$. This gives that ${}^*F\leq_\cR G$, and applying $-^*$ gives us the result. \end{proof}\end{lem}

We present the generalisations of \cite[Lemma 13, Corollary 14]{mazorchuk2011cell}.

\begin{lem}\label{MM1L13} For $F,K,H$ indecomposable 1-morphisms in $\cC$, $[FL_K:L_H]\neq 0$ implies $H\leq_\cL K$. If $H\leq_\cL K$, there exists some indecomposable 1-morphism $M\in \cC$ such that $[ML_K:L_H]\neq 0$.\end{lem}

\begin{cor}\label{MM1C14} Let $F,G,H\in S(\cC)$. If $L_F$ occurs in the top or socle of $HL_G$, then $F\in\ceL_G$.\end{cor}

\begin{psn}[{\cite[Proposition 17 a), b)]{mazorchuk2011cell}}]\label{MM1P17a} Let $\ceL$ be an $\cL$-cell of $\cC$ with domain $\tti$.\begin{itemize}
		\item There is a unique submodule $K=K_\ceL$ of $P_{\bbon_\tti}$ such that every simple subquotient of $P_{\bbon_\tti}/K$ is annihilated by any $F\in\ceL$ and such that $K$ has simple top $L_{G_\ceL}$ for some $G_\ceL\in\ceL$ such that $FL_{G_\ceL}\neq 0$ for any $F\in\ceL$. 
		\item For any $F\in\ceL$, $FL_{G_\ceL}$ has simple top $L_F$.\end{itemize}
\end{psn}

We call $G_\ceL$ as the \emph{Duflo involution} of $\ceL$.

\begin{psn}[{\cite[Proposition 17 c), e)]{mazorchuk2011cell}}]\label{MM1P17e} $G_\ceL, G_\ceL^*\in\ceL$.
\end{psn}

From this, if $\ceJ$ is strongly regular and $\ceL\subseteq \ceJ$, then $\ceL\cap\tensor[^*]{\ceL}{}=\{G_\ceL\}$. For the remainder of this subsection, we assume that $\ceJ$ is a maximal and strongly regular $\cJ$-cell.

\begin{lem}\label{MM1P34a} For $F,H\in\ceJ$, there exists some non-negative integer $\mbm_{F,H}$ such that $H^*\circ F\cong K^{\mbm_{F,H}}$, where $\{K\}=\ceR_{H^*}\cap\ceL_F$.
	\begin{proof} This is a direct consequence of $\ceJ$ being strongly regular and maximal and of $\cL$-cells being closed under indecomposable summands of left 1-composition and right cells under indecomposable summands of right 1-composition.\end{proof}\end{lem}

\begin{lem}[{\cite[Lemma 26]{mazorchuk2011cell}}]\label{MM1L26} For any $F\in S(\cC)$, $F^*\sim_\cJ F$.
\end{lem}

We move to consider the 2-representation $\mbM=\ol{\mbC_\ceL}$, the abelianisation of the cell 2-representation for some $\cL$-cell $\ceL$. We use $\ttP_F$, $\ttI_F$ and $\ttL_F$ to refer to projectives, injectives and simples in $\mbM$ respectively. For the remainder of this section and the following two we can assume, by quotienting $\cC$ by the 2-ideal generated by all $\id_F$ such that $F\not\leq\ceJ$ (for $\ceJ$ the $\cJ$-cell containing $\ceL$), that $\ceJ$ is the unique maximal $\cJ$-cell of $\cC$ (see \autoref{ResNotLeqJ} for more details). Further, we assume that $\ceJ$ is strongly regular. We give the generalisation of the first parts of \cite[Proposition 30]{mazorchuk2016isotypic}:

\begin{psn}\label{MM6P30a} The projective object $\ttP_F$ is injective for any $F\in\ceL$.
	\begin{proof} We mirror the proof for the above citation, with extra clarifying details. By adjunction, $$\Hom_{\ol{\ceM}}(\ttL_{G_\ceL},F^*\ttL_F)\cong\Hom_{\ol{\ceM}}(F\ttL_{G_\ceL},\ttL_F).$$ By \autoref{MM1P17a}, $\ttL_F$ is the simple top of $F\ttL_{G_\ceL}$ and hence the latter space is non-zero and one-dimensional. Since $\ttL_{G_\ceL}$ is simple, it follows that it injects into $F^*\ttL_F$.\par
		
		Let $I$ be an injective object in some $\ol{\mbC_\ceL}(\tti)$ and let $\ttL_K$ be one of its simple quotients with $K\in\ceL$. Then $\ttL_{G_\ceL}$ is a subquotient of the object $K^*I$ which is injective as $K^*$ is exact. Using \autoref{MM1P12} and the strong regularity of $\ceJ$ we have that $G_\ceL L_H=0$ unless $H\cong G_\ceL$. Therefore $G_\ceL\ttL_H=0$ unless $H\cong G_\ceL$. By \autoref{MM1P17a}, $G_\ceL\ttL_{G_\ceL}$ has simple top $\ttL_{G_\ceL}$ and $\ttL_{G_\ceL}$ appears in the top of the object $G_\ceL K^*I$, which is injective as $G_{\ceL}$ is exact. It follows that $\ttP_F$ appears as a quotient, and thus a direct summand, of $FG_{\ceL}K^*I$, which is injective as $F$ is exact. The result follows.\end{proof}\end{psn}

\begin{psn}\label{MM1P38a} For any $F\in\ceL$, $F^*\ttL_F\cong I_{G_\ceL}$.
	\begin{proof} This is a generalisation of the proof of \cite[Proposition 30 iii)]{mazorchuk2016isotypic}. We provide a few details to clarify for our situation. For any $H\in\ceL\cap\ol{\cC}(\tti,\tti)$, $\Hom_{\ol{\ceM}}(\ttL_H,F^*\ttL_F)\cong\Hom_{\ol{\ceM}}(F\ttL_H,\ttL_F)$. Since $\ceJ$ is strongly regular and $\ceL\cap\tensor[^*]{\ceL}{}=\{G_\ceL\}$, by \autoref{MM1P17a} it follows that $\Hom_{\ol{\cC(\tti,\ttj)}}(F\ttL_H,\ttL_F)\neq 0$ only if $H=G_\ceL$, and thus the same holds true in the quotient $\ol{\ceM}$. It also follows that $\dim\Hom_{\ol{\cC(\tti,\ttj)}}(\ttL_{G_\ceL}, F^*\ttL_F)=1$. Therefore, $\ttL_{G_\ceL}$ is the only simple that injects into $F^*\ttL_F$, and is further its simple socle. In particular, it follows that $F^*\ttL_F$ is indecomposable. But the proof of \cite[Proposition 30 ii)]{mazorchuk2016isotypic} generalises immediately to give that $F^*\ttL_F$ has a non-zero projective-injective summand, which must be $I_{G_\ceL}$ by construction, and the result follows.\end{proof}\end{psn}

\subsection{The Regularity Condition And Some Other Results}

We quote a classical result, \cite[Lemma 13]{mazorchuk2015transitive}.

\begin{psn}\label{MM5L13} Let $B$ be a finite dimensional $\bbk$-algebra and $G$ an exact endofunctor of $B-\modd$. Assume that $G$ sends each simple object to a projective object. Then $G$ is a functor isomorphic to tensoring with a projective bimodule (which we call a \emph{projective functor}).\end{psn}

\begin{psn}\label{MM6P1} Let $\ceJ$ be a strongly regular $\cJ$-cell, and let $\mbm:\ceJ\to\bbZ^+$ be defined as ${}^*F\circ F\cong H^{\oplus \mbm_F}\oplus K$, with no indecomposable direct summands of $K$ belonging to $\ceJ$. Then $\mbm$ is constant on $\cR$-cells of $\ceJ$.
	\begin{proof} This is a generalisation of \cite[Proposition 1]{mazorchuk2016isotypic}, and we mirror the proof therein. Let $\ceL$ be an $\cL$-cell in $\ceJ$ and let $\mbC_\ceL$ and $\ol{\mbC_\ceL}$ be the corresponding cell 2-representation and its abelianisation respectively. For $F,H\in\ceL$, by \autoref{MM1P12} and \autoref{MM1P17e} we can immediately derive that $F\ttL_H\cong\ttP_F$ if $H\cong G_\ceL$ and zero otherwise. Since every element of $\ceL$ has the same source object, we can apply \autoref{MM5L13} which gives that, for each $F\in\ceL$, $\ol{\mbC_\ceL}(F)$ is an indecomposable projective functor.\par
		
		For each $\ttj\in\cC$, let $A_\ttj$ denote the basic algebra such that $\ol{\mbC_\ceL}(\ttj)$ is equivalent to $A_\ttj$-$\modd$. Let $\{e_{\ttj 1},\dots,e_{\ttj {n_\ttj}}\}$ be a complete set of pairwise orthogonal primitive idempotents in $A_\ttj$. Then without loss of generality, each $\ol{\mbC_\ceL}(F)$ is the projective functor $A_\ttj e_{\ttj s}\otimes_\bbk e_{\tti 1} A_\tti\otimes_{A_\tti}-$ for some $s\in\{1,\dots,n_\ttj\}$.\par
		
		It follows from \autoref{MM6P30a} that $A_\ttj$ is self-injective when $\ceJ$ is maximal. There is thus some permutation $\sigma_\ttj\in S_{n_\ttj}$ such that $(e_{\ttj s} A_\ttj)^*\cong A_\ttj e_{\ttj \sigma_\ttj(s)}$. It follows that $\ol{\mbC_\ceL}(F^*)$ is the projective functor $A_\tti e_{\tti \sigma_{\tti}(1)}\otimes_\bbk e_{\ttj s} A_\ttj\otimes_{A_\ttj}-$. By taking the tensor product, $\mbm_{F^*}=\dim(e_{\tti 1} A_\tti e_{\tti \sigma_{\tti}(1)})$ which is independent of the choice of $F\in\ceL$. Since $F\mapsto F^*$ is a bijection from $\ceL$ to the $\cR$-cell of $\ceJ$ containing $G^*_\ceL$, $\mbm$ is constant on the $\cR$-cell. But by \autoref{MM1P17e} every $\cR$-cell contains a Duflo involution, and hence the result follows.\end{proof}\end{psn}

We can use this to derive the following:

\begin{lem}[{\cite[Lemma 29]{mazorchuk2016isotypic}}]\label{MM6P29} For $\ceJ$ a strongly regular maximal $\cJ$-cell in $\cC$ and any $F\in\ceL$ $FG_\ceL\cong F^{\oplus \mbm_{G_\ceL}}$.\end{lem}

\subsection{Restricting To Smaller 2-Categories}\label{ResNotLeqJ}
Let $\cC$ be a locally weakly fiat 2-category and let $\ceJ$ be a strongly regular $\cJ$-cell of $\cC$. We define the 2-ideal $\ceI_{\not\leq\ceJ}$ of $\cC$ to be generated by 2-morphisms $\id_F$ for $F\not\leq_\cJ \ceJ$. We also define $\ceI_{\leq\ceJ}$ to be the maximal 2-ideal of $\cC$ such that $\id_F\notin\ceI_{\leq\ceJ}$ for any $F\in\ceJ$. We define 2-categories $\cC_{\not\leq\ceJ}:=\cC/\ceI_{\not\leq\ceJ}$ and $\cC_{\leq\ceJ}:=\cC/\ceI_{\leq\ceJ}$. We also define $\cC^{\ceJ}_{\not\leq\ceJ}$ to be the sub-2-category of $\cC_{\not\leq\ceJ}$ closed under direct sums and isomorphisms and generated by the $\bbon_\tti$ for $\tti\in\cC$ and by $F\in\ceJ$.

\begin{lem}\label{Idnotleqsubset} $\ceI_{\not\leq\ceJ}\subseteq\ceI_{\leq\ceJ}$.
	\begin{proof} By the definition of $\ceI_{\leq\ceJ}$ it suffices to show that $\id_F\notin\ceI_{\not\leq\ceJ}$ for any $F\in\ceJ$. Assume otherwise for contradiction. Then for some $F\in\ceJ$ $\id_F=\sum_{k=1}^n f_km_kg_k$ where $m_k:G_k\to G_k$, $G_k\not\leq\ceJ$, $g_k:F\to G_k$, $f_k:G_k\to F$ for all $k$.\par
		
		Since $F$ is indecomposable, by a categorical variant of \cite[Corollary I.4.8]{assem2006elements}
		, for each $k$ either $f_km_kg_k$ is nilpotent or it is an automorphism. If $f_km_kg_k$ is nilpotent for all $k$, then as nilpotent morphisms form an ideal so is $\id_F$, a contradiction. Therefore there exists some $k$ such that $f_km_kg_k$ is an automorphism, say with inverse $h$. But then $(hf_km_k)g_k=\id_F\Rightarrow F\cong G_k$ and by construction $F\in\ceJ$ and $G_k\notin\ceJ$, a contradiction. The result follows.\end{proof}\end{lem}

\begin{psn}\label{JCellDescends} The image of $\ceJ$ remains a $\cJ$-cell in $\cC^{\ceJ}_{\not\leq\ceJ}$.

	\begin{proof} Let $\ceL$ be a left cell in $\ceJ$ and let $F\in\ceL$. We first note that, by \autoref{Idnotleqsubset}, $\id_F\notin\ceI_{\not\leq\ceJ}$ for any $F\in\ceJ$, and so $\End_{\cC_{\not\leq\ceJ}^{\ceJ}}(F)\neq 0$, and hence $F\neq 0$ in $\cC^{\ceJ}_{\not\leq\ceJ}$. For the Duflo involution $G_\ceL$ of $\ceL$ we know by strong regularity of $\ceJ$ that $FG_\ceL\cong F^{\oplus m}$ in $\cC^{\ceJ}_{\not\leq\ceJ}$ for some positive $m$. Since the composition of the injection 2-morphism $F\to FG_\ceL$ with the projection 2-morphism $FG_\ceL\to F$ is $\id_F$, the direct sum structure is preserved in $\cC^{\ceJ}_{\not\leq\ceJ}$ and so $G_\ceL\leq_\cL F$ in $\cC^{\ceJ}_{\not\leq\ceJ}$.\par
		
		Conversely, again by strong regularity of $\ceJ$, $F^*F\cong (G^*_\ceL)^{\oplus n}$ for some positive $n$ in $\cC^{\ceJ}_{\not\leq\ceJ}$, and consequently by a similar argument to above $F\leq_\cL G^*_\ceL$ in $\cC_{\not\leq\ceJ}$ and $\ceL$ is contained in an $\cL$-cell in $\cC_{\not\leq\ceJ}$. But it is immediate from the definitions that if $F\leq_\cL H$ in $\cC_{\not\leq\ceJ}$, then $F\leq_\cL H$ in $\cC$ and thus $\ceL$ is precisely an $\cL$-cell in $\cC_{\not\leq\ceJ}$. Applying adjunctions gives the corresponding result for right cells and the result follows.\end{proof}\end{psn}

For an $\cL$-cell $\ceL$ of $\ceJ$ we recall the finitary 2-representation $\mbN_\ceL:\cC\to\LA_\bbk^f$ given by $\mbN_\ceL(\ttj)=\add\{FX|F\in\coprod_{\ttk\in\cC}\cC(\ttk,\ttj),X\in\ceL\}$ and set $\ceN_\ceL=\coprod_{\ttj\in\cC}\mbN_\ceL(\ttj)$, with the class of morphisms $\operatorname{Ar}(\ceN_\ceL)$. \par

We define two 2-representations of $\cC_{\not\leq\ceJ}$. First let $\mbN^{\not\leq\ceJ}_\ceL:\cC_{\not\leq\ceJ}\to\LA_\bbk^f$ be defined by $\mbN^{\not\leq\ceJ}_\ceL(\ttj)= \add\{FX|F\in\coprod_{\ttk\in\cC}\cC_{\not\leq\ceJ}(\ttk,\ttj),X\in\ceL\}$. We define the 2-representation $\mbN_\ceL/\ceI_{\not\leq\ceJ}$ by setting $$(\mbN_\ceL/\ceI_{\not\leq\ceJ})(\ttj)=\mbN_\ceL(\ttj)/(\operatorname{Ar}(\ceN_\ceL)\cap\ceI_{\not\leq\ceJ}(\tti,\ttj)),$$ where $\tti$ is the source object of $\ceL$, with the obvious induced action of $\cC_{\not\leq\ceJ}$.

\begin{lem}\label{mbNJequiv} $\mbN_\ceL^{\not\leq\ceJ}$ and $\mbN_\ceL/\ceI_{\not\leq\ceJ}$ are equivalent as 2-representations of $\cC_{\not\leq\ceJ}$.
	\begin{proof} By construction there is a bijection between objects of $\ceN_\ceL^{\not\leq\ceJ}$ and $\ceN_\ceL/\ceI_{\not\leq\ceJ}$ and it suffices to show that for $F,G\in\mbN_\ceL(\ttj)$, $$\Hom_{\mbN_\ceL^{\not\leq\ceJ}}(F,G)\cong\Hom_{\mbN_\ceL/\ceI_{\not\leq\ceJ}}(F,G).$$
		But \begin{align*} \Hom_{\mbN_\ceL/\ceI_{\not\leq\ceJ}}(F,G) & = \Hom_{\ceN_\ceL}(F,G)/ (\Hom_{\ceN_\ceL}(F,G)\cap\ceI_{\not\leq\ceJ}(\tti,\ttj)\\
		& \cong\Hom_\cC(F,G)/(\Hom_\cC(F,G)\cap\ceI_{\not\leq\ceJ}(\tti,\ttj))\\
		& =\Hom_{\cC_{\not\leq\ceJ}}(F,G)\\
		& \cong\Hom_{\mbN_\ceL^{\not\leq\ceJ}}(F,G)
		\end{align*} as required.\end{proof}\end{lem}

By \autoref{JCellDescends} $\ceJ$ descends to a $\cJ$-cell of $\cC^\ceJ_{\not\leq\ceJ}$, which we will also denote by $\ceJ$. We can thus define the 2-representation $\mbN_\ceL^\ceJ$ of $\cC^\ceJ_{\not\leq\ceJ}$ in the standard fashion. We can consider $\mbN_\ceL^{\not\leq\ceJ}$ as a 2-representation of $\cC^\ceJ_{\not\leq\ceJ}$ by restriction. We define the 2-representation $(\mbN_{\ceL})^\ceJ$ of $\cC^\ceJ_{\not\leq\ceJ}$ as the full sub-2-representation of $\mbN_\ceL^{\not\leq\ceJ}$ generated by $F\in\ceJ$ and closed under isomorphism.

\begin{lem}\label{NequivBottom} $\mbN_\ceL^\ceJ$ is equivalent to $(\mbN_\ceL)^\ceJ$ as 2-representations of $\cC^\ceJ_{\not\leq\ceJ}$.
	\begin{proof} By construction, if we have 1-morphisms $F,G\in\cC^\ceJ_{\not\leq\ceJ}$ such that $F,G\in\ceN^\ceJ_\ceL$ and $F,G\in(\ceN_\ceL)^\ceJ$ then $$\Hom_{\ceN^\ceJ_\ceL}(F,G)\cong\Hom_{(\ceN_\ceL)^\ceJ}(F,G)$$ and thus it suffices to demonstrate an essential bijection between objects in the component categories of the 2-representations.\par
		
		If $F\in\mbN_\ceL^\ceJ(\ttj)$ is indecomposable, then $F$ is a direct summand of $GX$ for some $X\in\ceL$ and some $G\in\coprod_{\ttk\in\cC}\cC^\ceJ_{\not\leq\ceJ}(\ttk,\ttj)$. But then $G$ is a direct sum of elements of $\ceJ$ and thus $F\in\ceJ$. As we can also consider $G$ to be in $\coprod_{\ttk\in\cC}\cC_{\not\leq\ceJ}(\ttk,\ttj)$, it follows that $F\in(\mbN_\ceL)^\ceJ(\ttj)$.\par
		
		Conversely, let $F$ be an indecomposable object in $(\mbN_\ceL)^\ceJ(\ttj)$. Then $F\in\ceJ$ and $F$ is a direct summand of $GX$ for some $X\in\ceL$ and $G\in\coprod_{\ttk\in\cC}\cC_{\not\leq\ceJ}(\ttk,\ttj)$. Hence $X\leq_\cL F$. But by the definition of a strongly regular $\cJ$-cell, different left cells of $\ceJ$ are incomparable under $\leq_\cL$. Thus we must have that $F\sim_\cL X$ and $F\in\ceL$. But then $F$ is a direct summand of $\bbon_\ttj F$ and since $\bbon_\ttj\in\cC^\ceJ_{\not\leq\ceJ}$, the result follows.\end{proof}\end{lem}

We now consider the cell 2-representations. Let $\mbC_\ceL$ denote the 2-representation of $\cC$ corresponding to $\ceL$. This corresponds to the quotient of $\mbN_\ceL$ by the maximal $\cC$-stable ideal $\ceK_\ceL$ not containing $\id_F$ for any $F\in\ceL$. We define similar ideals $\ceK_\ceL^{\not\leq\ceJ}$ and $\ceK_\ceL^\ceJ$ of $\mbN_\ceL^{\not\leq\ceJ}$ and $\mbN_\ceL^\ceJ$ respectively. We let $\mbC_\ceL^{\not\leq\ceJ}$ and $\mbC_\ceL^\ceJ$ denote the respective cell 2-representations.

\begin{lem}\label{2RepIdSubset} $\mbN_\ceL(\ttj)\cap\ceI_{\not\leq\ceJ}(\tti,\ttj)\subseteq\ceK_\ceL$.
	\begin{proof} By the construction of $\ceK_\ceL$, it suffices to show that $\id_F\notin\mbN_\ceL(\ttj)\cap\ceI_{\not\leq\ceJ}(\tti,\ttj)$ for any $F\in\ceL$. But if $\id_F\in\mbN_\ceL(\ttj)\cap\ceI_{\not\leq\ceJ}(\tti,\ttj)$ for some $F\in\ceL\subseteq\ceJ$, then in particular $\id_F\in\ceI_{\not\leq\ceJ}(\tti,\ttj)$ for some $F\in\ceJ$, which we showed was a contradiction in the proof of \autoref{Idnotleqsubset}, and we are done.\end{proof}\end{lem}

\begin{psn}\label{Cell2RepResPartial} $\mbC_\ceL$ has a natural structure of a $\cC_{\not\leq\ceJ}$ 2-representation, and further is equivalent to $\mbC_\ceL^{\not\leq\ceJ}$ as 2-representations of $\cC_\ceL^{\not\leq\ceJ}$.
	
	\begin{proof} By \autoref{2RepIdSubset}, quotienting $\ceN_\ceL$ by $\ceK_\ceL$ factors through quotienting by $\ceN_\ceL\cap\ceI_{\not\leq\ceJ}$, giving the first statement. For the second, by \autoref{mbNJequiv} $\mbN_\ceL/\ceI_{\not\leq\ceJ}$ is equivalent to $\mbN_\ceJ^{\not\leq\ceJ}$. It thus suffices to show that $\ceK_\ceL^{\not\leq\ceJ}$ is the image under this equivalence of the image of $\ceK_\ceL$ in the quotient.\par
		
		Let $\sigma:\ceN_\ceL\to\ceN_\ceL/\ceI_{\not\leq\ceJ}$ denote the canonical quotient functor. It is straightforward to see that the preimage $\ceQ$ of $\ceK_\ceL^{\not\leq\ceJ}$ is a $\cC$-stable ideal of $\ceN_\ceL$. We will show that $\ceQ\subseteq\ceK_\ceL$. Assume for contradiction that $\id_F\in\ceQ$ for some $F\in\ceL$. Since $\id_F\notin \ceI_{\not\leq\ceJ}$ by the proof of \autoref{Idnotleqsubset}, this implies that $\id_F\in\ceK_\ceL^{\not\leq\ceJ}$, a contradiction. Therefore $\ceQ$ does not contain $\id_F$ for any $F\in\ceL$, and thus $\ceQ\subseteq\ceK_\ceL$. Hence $\ceK_\ceL^{\not\leq\ceJ}\subseteq\sigma(\ceK_\ceL)$. But by definition $\id_F\notin\sigma(\ceK_\ceL)$ for any $F\in\ceL$. Therefore by definition $\sigma(\ceK_\ceL)\subseteq\ceK_\ceL^{\not\leq\ceJ}$ and the result follows.\end{proof}\end{psn}

The cell 2-representation $\mbC_\ceL^{\not\leq\ceJ}$ has the structure of a 2-representation of $\cC^\ceJ_{\not\leq\ceJ}$ by restriction, and we can take the full sub-2-representation $(\mbC_\ceL)^\ceJ$ where the generating objects in the component categories are those in $\ceJ$.

\begin{psn}\label{Cell2RepRes1} $(\mbC_\ceL)^\ceJ$ is equivalent to $\mbC_\ceL^\ceJ$ as 2-representations of $\cC^\ceJ_{\not\leq\ceJ}$.
	
	\begin{proof} By \autoref{NequivBottom} it suffices to prove that the restriction $(\ceK_\ceL)^\ceJ$ of $\ceK_\ceL^{\not\leq\ceJ}$ to $(\mbN_\ceL)^\ceJ$ is equal to $\ceK_\ceL^\ceJ$. By construction $\id_F\notin(\ceK_\ceL)^\ceJ$ for any $F\in\ceL$, and thus $(\ceK_\ceL)^\ceJ\subseteq\ceK_\ceL^\ceJ$. It remains to show that $\ceK_\ceL^\ceJ\subseteq(\ceK_\ceL)^\ceJ$.\par
		
		Let $\ceQ$ denote the $\cC_{\not\leq\ceJ}$-stable ideal of $\mbN_\ceL^{\not\leq\ceJ}$ generated by $\ceK_\ceL^\ceJ$. We will show that $\ceQ\subseteq\ceK_\ceL^{\not\leq\ceJ}$. Assume for contradiction that $\id_F\in\ceQ$ for some $F\in\ceL$. Then using a similar component argument to previous proofs, we have that $\id_F=\beta \mbN_\ceL^{\not\leq\ceJ}(K)(\gamma)\alpha$, where $\gamma:G\to H$ is in $\ceK_\ceL^\ceJ$, $K\leq_\cJ\ceJ$, $\alpha:F\to KG$ and $\beta:KH\to F$. We immediately see that $F$ is a direct summand of $KG$ and $KH$.\par
		
		Let $S\in\ceJ$ be a 1-morphism such that $SF\neq 0$ in $\cC_{\not\leq\ceJ}$, which exists as $\ceJ$ is strongly regular. Then $\id_{SF}=S(\id_F)=S(\beta)SK(\gamma)S(\alpha)$. But for a indecomposable summand $V$ of $SF$, $V\geq_{\cJ} S$, and as $SK\neq 0$, it follows that $V\in\ceJ$. Hence by pre- and post-composing with injection and projection 2-morphisms it follows that $\id_V=\beta'SK(\gamma)\alpha'$ for some $\beta'$ and $\alpha'$. Without loss of generality $V\in\ceL$ (e.g. by taking $S=G_{\ceL}$), and by a similar argument to before every indecomposable summand of $SK$ is in $\ceJ$. Hence $\id_V\in \ceK_\ceL^\ceJ$, a contradiction. Hence $\ceQ\subseteq\ceK_\ceL^{\not\leq\ceJ}$, and thus $\ceK_\ceL^\ceJ\subseteq(\ceK_\ceL)^\ceJ$ and the result follows.
		
\end{proof}\end{psn}

\begin{cor}\label{Cell2RepRes2} The restriction of the cell 2-representation $\mbC_\ceL$ of $\cC$ to $\cC^\ceJ_{\not\leq\ceJ}$ is the corresponding cell 2-representation.
	
	\begin{proof} This is a direct consequence of combining \autoref{Cell2RepResPartial} and \autoref{Cell2RepRes1} given \autoref{JCellDescends}.\end{proof} \end{cor}

\subsection{$\ceJ$-Simple, $\ceJ$-Full and Almost Algebra 2-Categories}

In this section we will prove the following theorem:

\begin{thm}\label{MM1T43a} Let $\cC$ be a locally weakly fiat 2-category with $\ceJ$ a strongly regular $\cJ$-cell in $\cC$ and $\ceL$ a $\cL$-cell in $\ceJ$. Then for $F\in\ceL$ and $H\in\ceJ$, $H\ttL_F$ is either zero or an indecomposable projective in $\coprod\ol{\mbC_\ceL}(\tti)$.\end{thm}

We need a supplementary lemma.

\begin{lem}\label{MM1L38b} Let $F,H\in\ceJ$. Then $H\ttL_F$ is either zero or injective-projective in $\coprod_{\ttj\in\cC}\ol{\mbC_\ceL}(\ttj)$.
	\begin{proof} If $H\ttL_F\neq 0$, then by a variant of \autoref{MM1P12} $F^*$ and $H$ are in the same $\cL$-cell, and hence $H\ttL_F$ is a direct summand of $KF^*\ttL_F$ for some $K$ by strong regularity of $\ceJ$. By  \autoref{MM6P30a} and \autoref{MM1P38a} this is projective-injective.\end{proof}\end{lem}

\begin{proof}[Proof of \autoref{MM1T43a}] To prove the theorem given \autoref{MM1L38b} it suffices to prove that $H\ttL_F$ is indecomposable. This is an immediate generalisation of parts (iv) and (v) of \cite[Proposition 30]{mazorchuk2016isotypic}, using \autoref{MM6P1}, giving the result.\end{proof}

\begin{defn} Let $\cC$ be a locally finitary 2-category and $\mbM$ a 2-representation of $\cC$. We say that $\mbM$ is \emph{2-full} if for any 1-morphisms in $\cC$ the representation map $\Hom_{\cC}(F,G)\to\Hom_{\LX}(\mbM F,\mbM G)$ is surjective, where $\LX$ is the target 2-category of $\mbM$. For a $\cJ$-cell $\ceJ$ of $\cC$, we say $\mbM$ is \emph{$\ceJ$-2-full} if for every $F,G\in\ceJ$, the representation map is surjective.\end{defn}

\begin{defn} Let $\cC$ be a locally weakly fiat 2-category and let $\ceJ$ be a non-trivial $\cJ$-cell in $\cC$. $\cC$ is \emph{$\ceJ$-simple} if every non-trivial 2-ideal of $\cC$ contains $\id_F$ for some $F\in\ceJ$. \end{defn}

For this section, we will assume that $\cC$ is a locally weakly fiat 2-category with a unique non-trivial $\cJ$-cell $\ceJ$, that $\cC$ is $\ceJ$-simple, and that $\ceJ$ is strongly regular.\par

Let $\mbM=\ol{\mbC_\ceL}$, the abelian cell 2-representation for some left cell $\ceL\subseteq\ceJ$. We recall from \autoref{CAXDef} the construction $\cC_{A,X}$, a locally weakly fiat 2-category associated to a family of basic self-injective connected finite dimensional algebras $\{A_\tti|\tti\in I\}$ with certain subalgebras $X_\tti$. We give the generalisation of \cite[Theorem 13]{mazorchuk2016endomorphisms}.

\begin{thm}\label{MM3T13} $\cC$ is biequivalent to $\cC_{A,X}$ for some $A$ and $X$.\end{thm}
\begin{proof} We start by generalising the proof of \cite[Theorem 13]{mazorchuk2016endomorphisms}, with some extra detail for clarity. For $\tti\in\cC$, let $A_\tti$ be a basic connected finite dimensional $\bbk$-algebra such that $\mbM(\tti)$ is equivalent to $A_\tti$-$\modd$. Letting $Z_\tti$ be the centre of $A_\tti$, we define $X_\tti:=\mbM(\End_\cC(\bbon_\tti))\subseteq Z_\tti$. We can define an action of $\mbM(F)$ on $\cC_{A,X}$ for $F\in\ceJ$ using the equivalence between $\mbM(\tti)$ and $A_\tti$-$\modd$. Then by the definition of the cell 2-representation, each $\mbM(F)$ for $F\in\ceJ$ is a projective functor in $\End(\ceM)$, and since each $\mbM(\bbon_\tti)$ acts as the identity, this implies that $\mbM$ factors through $\cC_{A,X}$, and thus $\mbM$ corestricts to a 2-functor from $\cC$ to $\cC_{A,X}$. By construction $\mbM$ is surjective up to equivalence on objects (and indeed is bijective on objects).\par
	
	We will show that each $$\mbM_{\tti,\ttj}:\cC(\tti,\ttj)\to\cC_{A,X}(\tti,\ttj)$$ is an equivalence. Since $\cC$ is $\ceJ$-simple it follows that $\mbM_{\tti,\ttj}$ is faithful. To show that $\mbM_{\tti,\ttj}$ is essentially surjective on 1-morphisms, by construction a 1-morphism in $\cC_{A,X}$ is equivalent to tensoring with a projective $(A_\tti$-$A_\ttj)$-bimodule. In particular, any indecomposable $(A_\tti$-$A_\ttj)$-bimodule will take a simple module to either zero or to an indecomposable projective module. But by the construction of $\cC_{A,X}$ and \autoref{MM1T43a} these are precisely $\mbM(F)$ for $F$ indecomposable, giving essential surjectivity.\par
	
	By the construction of $\cC_{A,X}$, $\mbM$ is surjective when applied to $\End_\cC(\bbon_\tti)$. For any other hom-space $\Hom_{\cC}(F,G)$ with $F,G\neq\bbon_\tti$, by the definition of $\ceJ$-2-fullness, it clearly suffices to show that $\mbM$ is $\ceJ$-2-full. We will do this and show the remaining cases as a three step process: we will show that $$\Hom_\cC(G_\ceL,\bbon_\tti)\to\Hom_{\mbM(\tti)}(\mbM(G_\ceL),\mbM(\bbon_\tti))$$ is surjective for the Duflo involution $G_\ceL$, that this implies surjectivity for any $$\Hom_\cC(F,\bbon_\ttj)\to\Hom_{\mbM(\ttj)}(\mbM(F),\mbM(\bbon_\ttj))$$ and then derive that each $\mbM_{\tti,\ttj}$ is indeed full.\par
	
	These three steps are the generalisations of \cite[Theorem 9]{mazorchuk2016endomorphisms}(specifically the proof of that Theorem), \cite[Proposition 6]{mazorchuk2016endomorphisms} and \cite[Corollary 8]{mazorchuk2016endomorphisms} respectively. To give the generalisations, we thus first need to generalise \cite[Lemma 7]{mazorchuk2016endomorphisms}. While this is a 1-categorical statement, the way in which we use it requires us to give a slight generalisation and thus manipulate the proof:
	
	\begin{lem}\label{MM3L7} Let $A$ be a countable product of finite dimensional connected $\bbk$-algebras and let $e$ and $f$ be primitive idempotents of $A$. Assume that $F$ is an exact endofunctor of $A$-$\modd$ such that $FL_f\cong Ae$ and $FL_g=0$ for any other simple $L_g\not\cong L_f$. Then $F$ is isomorphic to the functor $F'$ given by tensoring with the bimodule $Ae\otimes_\bbk fA$, and moreover $$\Hom_{\LR_\bbk}(F,\id_{A\text{-}\modd})\cong\Hom_A(Ae,Af).$$
		\begin{proof} As $e$ and $f$ are primitive idempotents, they each belong to $A_e$ and $A_f$ for connected finite dimensional components $A_e$ and $A_f$ of $A$. Thus without loss of generality we can restrict $F$ to $A'$-$\modd$, where $A'=A_e\times A_f$, which is a connected finite dimensional algebra. Hence we can apply the original form of the lemma in \cite{mazorchuk2016endomorphisms} and the result follows.\end{proof}\end{lem}
	
	Using \autoref{MM3L7} and \autoref{MM1P17a} we can generalise the proof of \cite[Theorem 9]{mazorchuk2016endomorphisms} directly, since it is a local proof which does not use any properties of involutions. We give our version of that result:
	
	\begin{psn}\label{MM3P9} The representation map $$\Hom_\cC(G_\ceL,\bbon_\tti)\to\Hom_{\mbM(\tti)}(\mbM(G_\ceL),\mbM(\bbon_\tti))$$ is surjective.\end{psn}
	
	However, the proofs of \cite[Proposition 6, Corollary 8]{mazorchuk2016endomorphisms} do use that $-^*$ is an involution in that paper. For Proposition 6 we will give an adaptation of the whole proof, reworked to avoid the involution issues.
	
	\begin{psn}\label{MM3P6} Assuming that the representation map $$\Hom_\cC(F,\bbon_\ttj)\to\Hom_{\mbM(\ttj)}(\mbM(F),\mbM(\bbon_\ttj))$$ is surjective for $F=G_\ceL$ and $\ttj=\tti$, then it is surjective for any $F\in\ceJ$ and any corresponding $\ttj$.
		\begin{proof} Without loss of generality $F\in\cC(\ttj,\ttj)$. Let $H,K\in\ceL$ for some left cell $\ceL$ of $\ceJ$ and assume that $H,K\in\cC(\ttj,\ttk)$. By strong regularity $HK^*\cong aX$ for some $X\in\ceJ$ and some non-negative integer $a$. Since $\ceJ$ consists of a single $\cD$-cell, we can vary $H$ and $K$ over $\ceL$ to get any element of $\ceJ$, and in particular we can choose $H$ and $K$ such that $HK^*\cong aF$ for some non-negative integer $a$. To show that $HK^*\neq 0$, note $K^*\ttL_K\cong \ttI_{G_\ceL}\in\ol{\mbC_\ceL}(\ttj)$ by \autoref{MM1P38a} (which still applies to the cell 2-representation case) and further $H\ttI_{G_\ceL}\neq 0$ as $H\ttL_{G_\ceL}\neq 0$ by \autoref{MM1P17a}. It follows that $HK^*\neq 0$.\par
			
			Similarly, ${}^*KH\cong bG_\ceL$ for some non-negative integer $b$. In addition, since $-^*$ is an anti-auto-equivalence, it follows that $$\Hom_{\cC}(H,K)\cong\Hom_\cC(K^*,H^*).$$\par
			
			Applying adjunctions, we have that $$\Hom_\cC(H,K)\cong b\Hom_\cC(G_\ceL,\bbon_\tti),\qquad\qquad\Hom_\cC(K^*,H^*)\cong a\Hom_\cC(F,\bbon_\ttj).$$ Evaluating $\Hom_{\cC}(H,K)$ at $\ttL_{G_\ceL}$ is surjective, and thus $$\Hom_{\mbM(\ttj)}(H\ttL_{G_\ceL}, K\ttL_{G_\ceL})\cong b\Hom_{\mbM(\tti)}(G_\ceL \ttL_{G_\ceL}, \ttL_{G_\ceL}).$$ Applying \autoref{MM1P17a} gives that $G_\ceL \ttL_{G_\ceL}$ has simple top $\ttL_{G_\ceL}$.  Therefore the space $\Hom_{\mbM(\tti)}(G_\ceL \ttL_{G_\ceL}, \ttL_{G_\ceL})$ is one-dimensional and $$b=\dim\Hom_{\mbM(\ttj)}(H\ttL_{G_\ceL}, K\ttL_{G_\ceL}).$$ \par
			
			Let $\ttL_\ttj$ denote a multiplicity-free direct sum of all simple modules in $\mbM(\ttj)$. By adjunction $\Hom_{\mbM(\tti)}(K^*\ttL_\ttj, H^*\ttL_\ttj)\cong a\Hom_{\mbM(\ttj)}(F\ttL_\ttj,\ttL_\ttj).$ It follows from \autoref{MM1P12} that $K^*\ttL_Q\neq 0$ for $Q\in\ceL$ if and only if $Q$ is in the same right cell as $K$. But then by strong regularity $K\cong Q$. Thus by \autoref{MM1P38a} $K^*\ttL_\ttj\cong \ttI_{G_\ceL}$. By a similar argument $H^*\ttL_\ttj\cong \ttI_{G_\ceL}$ and the left side of the above isomorphism is isomorphic to $\End_{\mbM(\tti)}(\ttI_{G_\ceL})$.\par
			
			As $F$ is a direct summand of $HK^*$, it follows that $\ttL_K$ is the only summand of $\ttL_\ttj$ not annihilated by $F$. By \autoref{MM1T43a} $F\ttL_K$ is an indecomposable projective in $\mbM(\ttj)$, and thus by strong regularity we must have $F\ttL_K\cong \ttP_H$. Therefore $\dim\Hom_{\mbM(\ttj)}(F\ttL_\ttj,\ttL_\ttj)=1$ and $a=\dim\End_{\mbM(\tti)}(\ttI_{G_\ceL}).$
			
			From \autoref{MM1L38b} it follows that $H\ttL_{G_\ceL}$ is an indecomposable projective in $\ceM$ with simple top $\ttL_H$, and is thus isomorphic to $\ttP_H$. It follows that $\ttI_{G_\ceL}\cong \ttP_{G_\ceL^*}$. Using \autoref{MM3P9} we can apply \autoref{MM3L7} to $\Hom_\cC(G_\ceL,\bbon_\tti)$, and consequently $\Hom_\cC(G_\ceL,\bbon_\tti)\cong\End_{\mbM(\tti)}(\ttP_{G_\ceL})$. 
			
			We now show that $\dim\End_{\mbM(\tti)}(\ttP_{G_\ceL})=\dim\End_{\mbM(\tti)}(\ttP_{G^*_\ceL})$. Take some finite dimensional $\bbk$-algebra $A$ such that $\mbM(\tti)$ is equivalent to $A$-$\modd$. We can thus consider $\ttP_{G_\ceL}$ to be isomorphic to $Ae_{G_\ceL}$ for some idempotent $e_{G_\ceL}$ of $A$. Hence $$\End_{\mbM(\tti)}(\ttP_{G_\ceL})\cong\End_{A\text{-}\modd}(Ae_{G_\ceL})\cong e_{G_\ceL}Ae_{G_\ceL}.$$ But on the other hand $$ \End_{A\text{-}\modd}(Ae_{G_\ceL})\cong\End_{\modd\text{-}A}(e_{G_\ceL}A) \cong \End_{A\text{-}\modd}(Ae_{\sigma(G_\ceL)}),$$ where $\sigma$ is the permutation defined by the weakly fiat structure on $\cC$. But by this same structure $e_{\sigma(G_\ceL)}=e_{G_\ceL^*}$, i.e. $Ae_{\sigma(G_\ceL)}$ is isomorphic to $P_{G_\ceL^*}$.\par
			
			We thus have that $$\dim\Hom_\cC(G,\bbon_\tti)=\dim\End_{\mbM(\tti)}(P_{G_\ceL})=\dim\End_{\mbM(\tti)}(P_{G_\ceL^*}).$$ Using the above results and \autoref{MM3L7}, we have 
			\begin{align*} \dim\Hom_\cC(H,K) &=\dim\Hom_{\mbM(\ttj)}(H\ttL_{G_\ceL}, K\ttL_{G_\ceL})\dim\End_{\mbM(\tti)}(P_{G^*_\ceL}) \\
			&=\dim\Hom_{\mbM(\ttj)}(\ttP_H,\ttP_K)\dim\End_{\mbM(\tti)}(\ttP_{G^*_\ceL})\end{align*}
			and
			\begin{align*} \dim\Hom_\cC(K^*,H^*) &=\dim\Hom_\cC(F,\bbon_\tti)\dim\End_{\mbM(\tti)}(\ttI_{G_\ceL})\\
			&=\dim\Hom_\cC(F,\bbon_\tti)\dim\End_{\mbM(\tti)}(\ttP_{G^*_\ceL}).\end{align*}
			
			As $\cC$ is $\ceJ$-simple, $$\dim\Hom_\cC(F,\bbon_\ttj)\leq\dim\Hom_{\LR_\bbk}(\mbM(F),\mbM(\bbon_\ttj))$$ and applying \autoref{MM3L7} we see the latter is equal to $\dim\Hom_{\mbM(\ttj)}(\ttP_H,\ttP_K)$. Dividing by $\End_{\mbM(\tti)}(\ttP_{G^*_\ceL})$,
			\begin{align*}\dim\Hom_{\mbM(\ttj)}(\ttP_H,\ttP_K) &=\dim\Hom_\cC(F,\bbon_\tti)\\
			&\leq\dim\Hom_{\LR_\bbk}(\mbM(F),\mbM(\bbon_\ttj))\\
			&=\dim\Hom_{\mbM(\ttj)}(\ttP_H,\ttP_K)\end{align*}
			where the last equality follows by applying \autoref{MM3L7}. Therefore $$\dim\Hom_\cC(F,\bbon_\tti)=\dim\Hom_{\LR_\bbk}(\mbM(F),\mbM(\bbon_\tti)).$$
			
			Since the representation map is injective by $\ceJ$-simplicity of $\cC$, we get surjection and the result is proved.\end{proof}\end{psn}
		
	\begin{lem} Let $H,K\in\ceJ\cap\cC(\ttj,\ttk)$. If the representation map $$\Hom_\cC(G_\ceL,\bbon_\tti)\to\Hom_{\mbM(\tti)}(\mbM(G_\ceL),\mbM(\bbon_\tti))$$ is surjective, then so is the representation map $$\Hom_\cC(H,K)\to\Hom_\mbM(\mbM(H),\mbM(K)).$$
		\begin{proof} The proof is mostly an immediate generalisation of the one for \cite[Corollary 8]{mazorchuk2016endomorphisms}, except that ${}^*KH$ needs to be read for $K^*H$.\end{proof}\end{lem}
	
	From this it follows immediately that $\mbM$ is $\ceJ$-2-full and each $\mbM_{\tti\ttj}$ is full. Therefore the main theorem is proven.\end{proof}

\subsection{Reducing To The $\ceJ$-Simple Case}\label{JSimpSuff}

We will assume for this subsection that $\cC$ is a strongly regular locally weakly fiat 2-category.

\begin{lem}\label{MM5T18a} Let $\mbM$ be a simple transitive 2-representation of $\cC$. Then there exists some strongly regular $\cJ$-cell $\ceJ$ such that $\mbM$ factors over $\cC_{\not\leq\ceJ}$ and the restriction $\mbM_{\not\leq\ceJ}^\ceJ$ of this to $\cC^\ceJ_{\not\leq\ceJ}$ is still simple transitive.
	\begin{proof} The first part of the proof of \cite[Theorem 18]{mazorchuk2015transitive}, which generalises immediately to our setting, gives that $\mbM$ has an apex (as in \autoref{ApexDef}), which we denote by $\ceJ$. If $F$ is a 1-morphism of $\cC$ such that $F$ is not annihilated by $\mbM$, then consider the $\cJ$-cell $\ceK$ containing $F$. As $\cJ$-cells are partially ordered by $\leq_\cJ$, we must have that $\ceK\leq_\cJ \ceJ$. Therefore passing to $\cC_{\not\leq\ceJ}$ we can assume that $\ceJ$ is the unique maximal $\cJ$-cell of $\cC$.\par
		
		We note that $\mbM$ restricts to a 2-representation $\mbM^\ceJ_{\not\leq\ceJ}$ of $\cC^\ceJ_{\not\leq\ceJ}$. The argument given in the proof of \cite[Theorem 18]{mazorchuk2015transitive} generalises immediately to the locally weakly fiat case, and it follows that $\mbM^\ceJ_{\not\leq\ceJ}$ is simple transitive as required. \end{proof}\end{lem}

By construction, $\cC^\ceJ_{\not\leq\ceJ}$ has $\ceJ$ as its unique maximal $\cJ$-cell. We now give the following generalisation of \cite[Lemma 18]{mazorchuk2011additive}:

\begin{lem}\label{MM2L18} There is a unique 2-ideal $\cI$ of $\cC^\ceJ_{\not\leq\ceJ}$ such that $\cC^\ceJ_{\not\leq\ceJ}/\cI$ is $\ceJ$-simple.
\begin{proof} The proof of this result generalises immediately from that of \cite[Lemma 18]{mazorchuk2011additive}. \end{proof}\end{lem}

We let $\cC_\ceJ$ denote this quotient. We denote by $\mbM_{\ceJ}$ the restriction of $\mbM$ to $\cC_\ceJ$ (with $\ceM_\ceJ$ the corresponding coproduct category). We claim that $\mbM_\ceJ$ is a transitive 2-representation of $\cC_\ceJ$. To see this, we first note that since $\ceJ$ is the unique maximal $\cJ$-cell not annihilated by $\mbM$, it follows immediately that $\ker(\mbM)\subseteq\ceJ$. Second, let $N\in\ceM_\ceJ$, and let $F\in\ceJ$. Since $\mbM$ is a transitive 2-representation, any $M\in\ceM$ is isomorphic to a direct summand of $GFN$ for some 1-morphism $G\in\cC$. But by the construction of $\cJ$-cells, all indecomposable summands of $GF$ are in $\ceJ$, and thus $GF\in\cC_\ceJ$ and hence $\mbM_\ceJ$ is indeed transitive.\par

As $\cC_\ceJ$ is a $\ceJ$-simple category with a unique non-trivial two-sided ideal, it is biequivalent to $\cC_{A,X}$ for some $A$ and $X$. By a simple generalisation of a previous result, any simple transitive 2-representation of any $\cC_{A,X}$ is equivalent to a cell 2-representation. We thus have that $\mbM_\ceJ$ is equivalent to $(\mbC_{\ceL})_\ceJ$ for some $\cL$-cell $\ceL$ of $\ceJ$. We now provide a lemma that, along with \autoref{MM5T18a}, will allow us to generalise \cite[Theorem 18]{mazorchuk2015transitive}:

\begin{lem}\label{MM5T18b} If $\mbM$ is a simple transitive 2-representation of $\cC$ such that $\mbM_\ceJ$ is equivalent to some cell 2-representation of $\cC_\ceJ$, then $\mbM$ is equivalent to some cell 2-representation of $\cC$.
	\begin{proof} The second half of the proof of \cite[Theorem 18]{mazorchuk2015transitive} generalises immediately.\end{proof}\end{lem}

Hence we have:

\begin{thm}\label{LocWFBig} Any simple transitive 2-representation of $\cC$ is equivalent to a cell 2-representation of $\cC$.
	\begin{proof} This is an immediate consequence of applying \autoref{MM5T18b} to the result of \autoref{MM5T18a}.\end{proof}\end{thm}

\section{An Application: Cyclotomic 2-Kac--Moody Algebras}\label{2KMApp}

We present an application for this theory, which is a much larger class of cyclotomic 2-Kac--Moody algebras than have previously been accessible by this style of theory.

\subsection{Khovanov--Lauda--Rouquier Algebras}\label{KLRAlg}

We review some notation for Khovanov--Lauda--Rouquier (KLR) algebras. The original constructions in this subsection come from \cite{khovanov2009diagrammatic} and \cite{rouquier20082}, but the specific notation below is based on the notation found in \cite{kang2012categorification} and \cite{hong2002introduction}, which works better for our setup. Let $$(A,P,\Pi=\{\alpha_i|i\in I\},P^{\vee}, \Pi^{\vee}=\{\alpha_i^{\vee}|i\in I\})$$ be a Cartan datum. We denote the set of dominant integral weights by $P^{+}$. Given an algebraically closed field $\bbk$, let $R(n)$ denote the KLR algebra of degree $n$ over $\bbk$ associated to the above Cartan datum, as defined in e.g. \cite[Section 3]{kang2012categorification} Section 3.\par

We divide $I^n$ into disjoint subsets by choosing some weight $\beta=\sum_{i=1}^n \alpha_i$ and setting $I^\beta=\{\nu\in I^n|\alpha_{\nu_1}+\dots+\alpha_{\nu_n}=\beta\}$. Notating $e(\beta)=\sum_{\nu\in I^\beta} e(\nu)$, we thus define $R(\beta)=R(n)e(\beta)$. Finally, we define $$e(\beta,i)=\sum_{\nu\in I^{\beta+\alpha_i},\:\nu_{n+1}=i} e(\nu)\in R(\beta+\alpha_i),$$ $$e(i,\beta)=\sum_{\nu\in I^{\beta+\alpha_i},\nu_1=1} e(\nu)\in R(\beta+\alpha_i).$$ We can apply a $\bbZ$-grading on $R(n)$ via $\deg e(\nu)=0$, $\deg x_ke(\nu)=(\alpha_{\nu_k}|\alpha_{\nu_k})$ and $\deg \tau_le(\nu)=-(\alpha_{\nu_l}|\alpha_{\nu_{l+1}})$. We denote the grading shift of degree $i$ of a module $M$ by $M\llbracket i\rrbracket$, where $(M\llbracket i\rrbracket)_j=M_{j-i}$.\par

For this paper we will be working with cyclotomic KLR algebras. Let $\Lambda\in P^+$. For $1\leq k\leq n$, we can define $a^\Lambda(x_k)=\sum_{\nu\in I^n} x_k^{\langle h_{\nu_k}, \Lambda\rangle}e(\nu)\in R(n)$.

\begin{defn} The \emph{cyclotomic Khovanov--Lauda--Rouquier algebra} $R^\Lambda(\beta)$ of weight $\beta$ at $\Lambda$ is defined as the quotient algebra $R^\Lambda(\beta)=\frac{R(\beta)}{R(\beta)a^\Lambda(x_1)R(\beta)}$, with $R^\Lambda(0)=\bbk$. \end{defn}

\begin{defn}\label{EFDefns} For each $i\in I$, we define functors \[E_i^\Lambda:R^\Lambda(\beta+\alpha_i)\text{-}\Mod\to R^\Lambda(\beta)\text{-}\Mod\]\[F_i^\Lambda:R^\Lambda(\beta)\text{-}\Mod\to R^\Lambda(\beta+\alpha_i)\text{-}\Mod\] by \[E_i^\Lambda(N)=e(\beta,i)N=e(\beta,i)R^\Lambda(\beta+\alpha_i) \otimes_{R^\Lambda(\beta+\alpha_i)} N \]\[E^\Lambda_i(f)=e(\beta,i)f=\id_{e(\beta,i)R^\Lambda(\beta+\alpha_i)}\otimes f\] and \[F_i^\Lambda(M)=R^\Lambda(\beta+\alpha_i)e(\beta,i)\otimes_{R^\Lambda(\beta)}M\] \[F_i^\Lambda(g)=\id_{R^\Lambda(\beta+\alpha_i)e(\beta,i)}\otimes g.\] We notate $e(\beta,i)R^\Lambda(\beta+\alpha_i)$ as $\mathfrak{e}_i$ and $R^\Lambda(\beta+\alpha_i)e(\beta,i)$ as $\mathfrak{f}_i$.\end{defn}

As noted in \cite[Section 4]{kang2012categorification}, we often apply a grading shift of $\llbracket1-\langle h_i,\Lambda-\beta\rangle\rrbracket$ (or scalar multiples thereof) to $\mathfrak{e}_i$, which will allow us later to consider adjunctions in a fashion that are compatible with the grading, which will be very useful in \autoref{gr2KM}.

We now give a useful result from \cite{kang2012categorification}:

\begin{thm}[{\cite[Theorems 5.1, 5.2]{kang2012categorification}}]\label{FEtoEF} Let $\lambda=\Lambda-\beta$. We have the following isomorphisms of $R^\Lambda(\beta)$ modules. For $j\neq i$, there exists a natural isomorphism $F^\Lambda_j E_i^\Lambda\llbracket-(\alpha_i|\alpha_j)\rrbracket\xrightarrow{\sim} E_i^\Lambda F_j^\Lambda$. If $i=j$ we have one of two cases:\begin{itemize}
	\item If $\langle h_i,\lambda\rangle\geq 0$, $$F_i^\Lambda E_i^\Lambda\llbracket -(\alpha_i|\alpha_i)\rrbracket \oplus\bigoplus_{k=0}^{\langle h_i,\lambda\rangle-1}\bbon_\lambda\llbracket k(\alpha_i|\alpha_i)\rrbracket \xrightarrow{\sim} E_i^\Lambda F_i^\Lambda.$$
	\item If $\langle h_i,\lambda\rangle\leq 0$, $$F_i^\Lambda E_i^\Lambda\llbracket -(\alpha_i|\alpha_i)\rrbracket\xrightarrow{\sim}\bigoplus_{k=0}^{-\langle h_i,\lambda\rangle -1} \bbon_\lambda\llbracket -(k-1)(\alpha_i|\alpha_i)\rrbracket\oplus E_i^\Lambda F_i^\Lambda.$$\end{itemize}\end{thm}

\subsection{Cyclotomic 2-Kac--Moody Algebras}\label{2KMADefn}

Let $U_q(\Lg)$ be the quantum group associated to a Kac--Moody algebra with $\dot{U}_q(\Lg)$ Lusztig's idempotent completion, let $\Lambda\in P^+$ and let $V(\Lambda)$ be its irreducible highest weight module (see e.g. \cite[Sections 2, 3]{hong2002introduction}for details). We then categorify the endomorphisms of its highest-weight module, following the original definition in \cite{websterknot}, as the 2-category $\cU_\Lambda$:

\begin{itemize}
	\item The objects of $\cU_\Lambda$ are the weights $\lambda$ such that $V(\Lambda)_{\lambda}\neq\{0\}$. Writing $\lambda=\Lambda-\beta$ for some $\beta$, we identify these with (small categories equivalent to) the module categories $R^\Lambda(\beta)$-$\proj$.
	
	\item The 1-morphisms of $\cU_\Lambda$ are direct summands of direct sums of the identity 1-morphisms and of compositions of 1-morphisms isomorphic to functors formed by tensoring with tensor products of (grade-shifts of) the $\mathfrak{e}_i$ and $\mathfrak{f}_i$ as defined in \autoref{EFDefns} (at any weight $\beta$ below $\Lambda$). Following that section, we denote the functor given by tensoring with $\mathfrak{e}_i\llbracket g\rrbracket$ by $E_i^\Lambda\llbracket g\rrbracket$ and the functor given by tensoring with $\mathfrak{f}_i\llbracket g\rrbracket$ as $F_i^\Lambda\llbracket g\rrbracket$.
	
	\item The 2-morphisms are the bimodule homomorphisms between the bimodules that correspond to the 1-morphisms. This implies that, for any $\cU_\Lambda(\lambda,\mu)$, the spaces of 2-morphisms are finite dimensional.\end{itemize}

\begin{defn}\label{C2KMDefn} We call this construction the \emph{cyclotomic 2-Kac--Moody category of weight $\Lambda$} associated to a Kac--Moody algebra $U(\Lg)$.\end{defn}

\begin{thm}\label{2KMFin} $\cU_\Lambda$ is a locally fiat 2-category.
	\begin{proof} We begin by showing that $\cU_\Lambda$ is a locally finitary 2-category. To do this, we wish to show $\cU_\Lambda(\lambda,\mu)\in\LA_\bbk^f$ for all weights $\lambda$ and $\mu$. We already have that the (2-)morphisms form a finite dimensional space and as $R^\Lambda(\beta)$ is indecomposable for all $\beta$, $\bbon_\lambda$ is indecomposable for all $\lambda$. It thus remains to show that there are only finitely many isomorphism classes of indecomposable objects. The objects for this category are generated by products and direct summands of the $E_i^\Lambda\bbon_{\zeta}$ and the $F_i^\Lambda\bbon_{\zeta}$ for arbitrary weights $\zeta$. Let $Q:\lambda\to\mu$ be a general 1-morphism. We wish to show that $Q\in\add(\{F^\Lambda_{i_1}\dots F^\Lambda_{i_l}\bbon_\lambda\}\cup\{F^\Lambda_{j_1}\dots F^\Lambda_{j_m}E^\Lambda_{k_1}\dots E^\Lambda_{k_n}\bbon_\lambda\}\cup\{\delta_{\lambda\mu}\bbon_\lambda\})$ for some choice of $i_x$, $j_y$ and $k_z$.\par
		
		If $Q$ is of the form $M_1E^\Lambda_iF^\Lambda_jM_2\xrightarrow{\sim} M_1E^\Lambda_iF^\Lambda_j\bbon_\epsilon M_2$ for some products $M_1$ and $M_2$ and some weight $\epsilon$, then if $i\neq j$, $M_1E^\Lambda_iF^\Lambda_jM_2\xrightarrow{\sim} M_1F^\Lambda_jE^\Lambda_iM_2$. If $i=j$ then we can use \autoref{FEtoEF} and the fact that composition distributes over direct sums to get one of the two following cases, depending on $\epsilon$:
		\begin{itemize}
			\item If $\langle h_i,\epsilon\rangle\geq 0$, $$M_1E^\Lambda_iF^\Lambda_iM_2\xrightarrow{\sim} M_1F^\Lambda_iE^\Lambda_iM_2\llbracket -(\alpha_i|\alpha_i)\rrbracket\oplus\bigoplus_{k=0}^{\langle h_i,\epsilon\rangle -1} M_1M_2\llbracket k(\alpha_i|\alpha_i)\rrbracket.$$
			\item  If $\langle h_i,\epsilon\rangle\leq 0$, \begin{align*} M_1E^\Lambda_iF^\Lambda_iM_2\oplus\bigoplus_{k=0}^{-\langle h_i,\epsilon\rangle -1} M_1M_2\llbracket -(k-1)(\alpha_i|\alpha_i)\rrbracket\\
			\xrightarrow{\sim} M_1F^\Lambda_iE^\Lambda_iM_2\llbracket -(\alpha_i|\alpha_i)\rrbracket.\end{align*}\end{itemize}
		In the second case, $Q=M_1E^\Lambda_iF^\Lambda_iM_2\in\add(M_1F^\Lambda_iE^\Lambda_iM_2)$, while in the first case $Q\in\add(\{M_1F^\Lambda_iE^\Lambda_iM_2,M_1M_2\})$.\par
		
		Let $R$ be some formal product of the $E_i^\Lambda$ and the $F_i^\Lambda$, which we will notate as $R=T_{j_0}E^\Lambda_{i_1}T_{j_1}E^\Lambda_{i_2}\dots E^\Lambda_{i_n}T_{j_n}$, where each $T_{j_k}$ is a possibly empty product of the $F^\Lambda_i$. If $T_{j_q}=F^\Lambda_{q_1}\dots F^\Lambda_{q_l}$, we let $|T_{j_q}|=l$. Define the finite non-negative integer $\Len(R)=\sum_{m=1}^n \sum_{q=m}^n |T_{j_q}|$, the length of $R$. Note that if $T_{j_k}=0$ for all $k>0$, which corresponds to $R=F^\Lambda_{i_1}\dots F^\Lambda_{i_n}E^\Lambda_{j_1}\dots E^\Lambda_{j_m}$ for $m,n\geq 0$, then $\Len(R)=0$. Further, $$\Len(Q)=\Len(M_1E^\Lambda_iF^\Lambda_jM_2)=\Len(M_1F^\Lambda_jE^\Lambda_iM_2)+1$$ and $\Len(Q)>\Len(M_1M_2)$. Finally, if $\Len(Q)>0$, there exists a subproduct $E^\Lambda_iF^\Lambda_j$ somewhere in $Q$, and we can apply one of the above operations to it. If $\Len(Q)>0$, $Q$ is thus contained in the additive closure of finitely many 1-morphisms of strictly lesser length, and as length is non-negative, proceeding recursively will terminate in finite time, giving the claim.\par
		
		We claim that there are only finitely many $F^\Lambda_{i_1}\dots F^\Lambda_{i_l}\bbon_\lambda:\lambda\to\mu$ and only finitely many $F^\Lambda_{j_1}\dots F^\Lambda_{j_m}E^\Lambda_{k_1}\dots E^\Lambda_{k_n}\bbon_\lambda:\lambda\to\mu$. For the first case, given $F^\Lambda_i\bbon_\lambda$ takes $\lambda$ to $\lambda-\alpha_i$, a simple combinatorial argument shows there can only be finitely many of them from $\lambda$ to $\mu$. \par
		
		For the $F^\Lambda_{j_1}\dots F^\Lambda_{j_m}E^\Lambda_{k_1}\dots E^\Lambda_{k_n}\bbon_\lambda$, write $\lambda=\Lambda-\sum_i b_i\alpha_i$, where all the $b_i$ are non-negative. Then as the $E^\Lambda_i$ move up the poset of weights, by a similar argument to the previous one, there are only finitely many non-zero products $E^\Lambda_{k_1}\dots E^\Lambda_{k_n}\bbon_\lambda:\lambda\to\delta$ with $\Lambda\geq\delta\geq\lambda$. Then by another similar argument, for any such $\delta$ there are only finitely many possible products $F^\Lambda_{j_1}\dots F^\Lambda_{j_m}\bbon_\delta:\delta\to\mu$. Thus there are in total only finitely many $F^\Lambda_{j_1}\dots F^\Lambda_{j_m}E^\Lambda_{k_1}\dots E^\Lambda_{k_n}\bbon_\lambda:\lambda\to\mu$ as we required.\par
		
		Each of these morphisms has a finite number of indecomposable direct summands. We thus find that $\cU(\lambda,\mu)$ can only have finitely many isomorphism classes of indecomposable objects. This gives the locally finitary structure.\par
		
		For the fiat structure, we prove that the 1-morphisms $E_i$ and $F_i$ have adjoints, and the other 1-morphisms will be an immediate consequence through composition. We claim that the adjoint of $E^\Lambda_i\bbon_\lambda\llbracket z\rrbracket$ is $F^\Lambda_i\bbon_{\lambda-\alpha_i}\llbracket z-\frac{(\alpha_i,\alpha_i)}{2}(1+\langle \alpha_i,\lambda\rangle)\rrbracket$ and that the adjoint of $F^\Lambda_i\bbon_\lambda\llbracket z\rrbracket$ is $E^\Lambda_i\bbon_{\lambda+\alpha_i}\llbracket z-\frac{(\alpha_i,\alpha_i)}{2}(1-\langle \alpha_i,\lambda\rangle)\rrbracket$.\par
		
		To see this, consider $F_i^\Lambda=R^\Lambda(\beta+\alpha_i)e(\beta,i)\otimes_{R^\Lambda(\beta)}-$ and ignore grading for the moment. The right adjoint of this is therefore $\Hom_{R^\Lambda(\beta+\alpha_i)}(R^\Lambda(\beta+\alpha_i)e(\beta,i),-)$.  But since $R^\Lambda(\beta+\alpha_i)e(\beta,i)$ is projective over $R^\Lambda(\beta+\alpha_i)$, this is isomorphic to $$\Hom_{R^\Lambda(\beta+\alpha_i)}(R^\Lambda(\beta+\alpha_i)e(\beta,i),R^\Lambda(\beta+\alpha_i))\otimes_{R^\Lambda(\beta+\alpha_i)}-.$$ This is isomorphic to $$\Hom_\bbk(R^\Lambda(\beta+\alpha_i)e(\beta,i),\bbk)\otimes_{R^\Lambda(\beta+\alpha_i)}-$$ because $R^\Lambda(\beta+\alpha_i)$ is symmetric. But by another application of this symmetric property, this is then isomorphic to $e(\beta,i)R^\Lambda(\beta+\alpha_i)\otimes_{R^\Lambda(\beta+\alpha_i)}-=E^\Lambda_i$. Finally, the grading is a consequence of the comment before \autoref{FEtoEF}.\end{proof}\end{thm}

Further, this 2-category will turn out to be strongly regular. However, to prove this we need to extend our definition of a cyclotomic 2-Kac--Moody algebra to a wider setup. This definition is a generalisation of a construction from \cite[Section 7.2]{mazorchuk2015transitive}.

\begin{defn} Choose a set of positive weights $\ul{\Lambda}=\{\Lambda_1,\dots,\Lambda_n\}\subseteq P^+$. Without loss of generality we assume that $\Lambda_i\not\leq\Lambda_j$ for $i\neq j$. We define a 2-category $\cU_{\ul{\Lambda}}$, the \emph{truncated cyclotomic 2-Kac--Moody algebra}, as follows:\begin{enumerate}
		\item The objects of $\cU_{\ul{\Lambda}}$ are ordered pairs $(\beta,i)$ where $\beta\in Q^+$ and $1\leq i\leq n$, modulo an equivalence relation where $(\beta,i)\sim(\gamma,j)$ if $\Lambda_i-\beta=\Lambda_j-\gamma$.
		\item The 1-morphisms of $\cU_{\ul{\Lambda}}$ are the additive closure of (grade shifts of) the identity 1-morphisms and morphisms of the form $E^\Lambda_i$ and $F^\Lambda_i$, as in the cyclotomic 2-Kac--Moody algebra case, with identical relations to that situation.
		\item The 2-morphisms are identical to the single weight definition (\autoref{C2KMDefn}).
\end{enumerate}\end{defn}

We note that this 2-category is well-defined. In addition, using the interchange structure and \autoref{2KMFin} for distinct $F^\Lambda_i$ and $E^\Lambda_j$, if $(\beta, i)$ and $(\gamma, j)$ are objects of $\cC$ such that there does not exist $\Lambda_k$ with both $\Lambda_i-\beta\leq \Lambda_k$ and $\Lambda_j-\gamma\leq \Lambda_k$, then ${\cU_{\ul{\Lambda}}}((\beta, i),(\gamma, j))={\cU_{\ul{\Lambda}}}((\gamma, j),(\beta, i))=0$. We also define the notation $\ul{\Lambda}^p$ as $\ul{\Lambda}^p=\{\lambda|\exists i, \exists\beta, \lambda=\Lambda_i-\beta\}$, the set of weights below at least one of the $\Lambda_i$. We also note that if $\ul{\Lambda}=\{\Lambda_1\}$, then $\cU_{\ul{\Lambda}}=\cU_\Lambda$ as previously defined.\par

That this 2-category is locally weakly fiat follows immediately from the above considerations, since the internal adjoint 1-morphism and adjunction 2-morphisms will remain identical to the traditional case. We will combine this with the following result:

\begin{psn}\label{2KMSR} For any $\ul{\Lambda}$, $\cU_{\ul{\Lambda}}$ is strongly regular.
	
	\begin{proof} We mirror the proof for \cite[Theorem 21]{mazorchuk2015transitive}. We first consider the $\cJ$-cell $\ceJ_{\Lambda_i}$ containing $\bbon_{\Lambda_i}$ for some $\Lambda_i\in\ul{\Lambda}$. If we quotient out by the maximal 2-ideal in $\cU_{\ul{\Lambda}}$ which contains $\id_{\bbon_{\Lambda_i}}$ but not any identity 2-morphisms for a 1-morphism not in $\ceJ_{\Lambda_i}$, the resulting 2-category is equivalent to one of the form $\cU_{\ul{\Theta}}$, where $\ul{\Theta}$ is the unique set of highest weights such that $\ul{\Theta}^p=\ul{\Lambda}^p\setminus\{\Lambda_i\}$ (see \cite[Section 9]{doty2017cellular} for more details). It is thus sufficient to prove that $\ceJ_{\Lambda_i}$ is strongly regular. \par
		
		Let $\ceL$ denote the $\cL$-cell of $\bbon_{\Lambda_i}$. From the proof of \autoref{2KMFin}, any element of $\ceL$ is in the additive closure of 1-morphisms of the form $F^\Lambda_1\dots F^\Lambda_n$ and $F^\Lambda_1\dots F^\Lambda_nE^\Lambda_1\dots E^\Lambda_m$. But since any element of $\ceL$ must have source object $\Lambda_i$ and any morphism of the form $F^\Lambda_1\dots F^\Lambda_nE^\Lambda_i\dots E^\Lambda_m$ with source object $\Lambda_i$ must necessarily be zero (as $\Lambda_i$ is a highest weight in the 2-category), it follows that $\ceL$ consists of direct summands of products of the $F_i$.\par
		
		Let $L$ be an indecomposable object in $R^{\Lambda_i}_0$-$\proj$. Since $R^{\Lambda_i}_0\cong\bbk$, we have that $L\cong\bbk$. By \cite[Theorem 5.7]{rouquier20082} and \cite[Theorem 4.4]{varagnolo2011canonical}, the mapping that takes an $F\in\ceL$ to $FL$ induces a bijection between $\ceL$ and the set of isomorphism classes of indecomposable objects in $\prod_{n\geq 0}{R^{\Lambda_i}_n}\text{-}\proj$. We define two algebras, $A=\bigoplus_{n\geq 0} R^{\Lambda_i}_n$ and $B=\bigoplus_{n\geq 1} R^{\Lambda_i}_n$. Since every element $X\in\ceL$ can be expressed as $X\bbon_{\Lambda_i}$ and as $\bbon_{\Lambda_i}M=0$ for any $M\in B$-$\operatorname{proj}$, it follows that $XM=0$ for any $X\in\ceL$. We now consider the (projective) abelianisation $\ol{\mbC_{\ceL}}$ of the cell 2-representation for $\ceL$.\par
		
		By the construction of the abelianisation, $\ol{\mbC_\ceL}(X)$ can be considered as a functor from $\bbk$-$\modd$ to $R^{\Lambda_i}(\beta)$-$\modd$ for some positive weight $\beta$. This can consequently be considered as an endofunctor of $\bbk\times R^{\Lambda_i}(\beta)$-$\modd$. Further, since the only projective it is non-zero on is $L$, which it must take to an indecomposable projective, it follows from \cite[Lemma 13]{mazorchuk2015transitive} that $\ol{\mbC_\ceL}(X)$ is an indecomposable projective endofunctor. By consideration of sub-categories of the domain and of the range where this functor acts trivially or does not map to respectively, we can indeed say that $\ol{\mbC_\ceL}(X)$ is an indecomposable projective functor from $\bbk$-$\modd$ to $A$-$\modd$. But by this projectivity, for any $Y\in\ceL$, $\ol{\mbC_\ceL}(X\circ Y^*)$ is also indecomposable.\par
		
		We claim that this implies that $X\circ Y^*$ is itself indecomposable. For assume that $X\circ Y^*\cong V\oplus W$ for non-zero $V$ and $W$. Then without loss of generality $\ol{\mbC_\ceL}(W)=0$. But since $\ceJ$ is a maximal 2-sided cell, we must have for every indecomposable summand $W'$ of $W$ that $W'\in\ceJ$. But then by construction, $\ol{\mbC_\ceL}(W')\neq 0$ and hence $\ol{\mbC_\ceL}(W)\neq 0$, a contradiction. The claim follows.\par
		
		Hence the set $\{X\circ Y^*\}$ is a set of indecomposables that forms a $\cD$-cell by construction and hence by \autoref{LFJIsD} is a $\cJ$-cell that contains $\bbon_{\Lambda_i}$, and thus is equal to $\ceJ$. Now fixing $X$ and varying $Y$ clearly gives a $\cR$-cell in $\ceJ$, and fixing $Y$ and varying $X$ gives an $\cL$-cell, and therefore this process must exhaust all such $\cL$- and $\cR$-cells. In particular, the intersection of any $\cL$-cell with any $\cR$-cell is thus a unique element. Thus $\ceJ$ is strongly regular, and the result follows. \end{proof}\end{psn}

\begin{thm}\label{2KMMain} Every simple transitive 2-representation of a truncated cyclotomic 2-Kac--Moody algebra is equivalent to a cell 2-representation.
	\begin{proof} By \autoref{2KMFin} and \autoref{2KMSR}, a truncated cyclotomic 2-Kac--Moody algebra is a strongly regular locally weakly fiat 2-category. Therefore applying \autoref{LocWFBig} gives the result immediately.\end{proof}\end{thm}

\section{A Specialisation To Graded 2-Categories}

We now move to considering the graded setup. It is worth pointing out that we are not considering the full generalisation of allowing infinite-dimensional hom-spaces of 2-morphisms with finite dimensional graded components, as is common in the literature; rather, we are considering locally finitary 2-categories which also have a graded structure (as defined in the following sections). For the rest of the paper, we always take $G$ to be a countable abelian group unless otherwise stated.

\subsection{Initial Definitions}
We start by defining $G$-graded (2-)categories and their $G$-envelopes, following the ideas in \cite[Section 3.5]{brundan2017categorical}. By a $G$-graded ($\bbk$-)vector space, we mean a vector space $A$ with a direct sum decomposition $A=\bigoplus_{g\in G} A_g$. There is a forgetful functor from the category of graded $\bbk$-vector spaces to the category of $\bbk$-vector spaces that forgets the grading. When we refer to isomorphisms of $G$-graded vector spaces, we mean isomorphisms in the latter category under the forgetful functor, unless otherwise indicated. If $A$ is a $G$-graded vector space and $g\in G$, we denote by $A\llbracket g\rrbracket$ the $G$-graded vector space isomorphic to $A$ such that $A\llbracket g\rrbracket_h= A_{h-g}$.

\begin{defn} Let $A$ be a $G$-graded $\bbk$-algebra. We say that $A$ is \emph{$G$-graded-finite dimensional} if it has a $G$-grading $A=\bigoplus_{g\in G} A_g$ such that each $A_g$ is finite dimensional as a $\bbk$-vector space. In particular, this implies that $A_0$ is a finite dimensional $\bbk$-algebra.\end{defn}

\begin{defn} The category $\bbk$-$\Mod_G$ has as objects the $G$-graded $\bbk$-vector spaces, and as morphisms finite linear combinations of homogeneous linear maps of arbitrary degree $g\in G$. We define the full subcategory $\bbk$-$\Mod_G^{\operatorname{gf}}$ to contain the $G$-graded $\bbk$-vector spaces that are $G$-graded-finite dimensional. We let $\bbk$-$\Mod_{G,0}$ and $\bbk$-$\Mod_{G,0}^{\operatorname{gf}}$ denote the subcategories of the above categories with the same objects but with morphisms only those homogeneous linear maps of degree zero.\end{defn}

\begin{defn} We define a \emph{$G$-graded category} to be a category enriched over $\bbk$-$\Mod_{G,0}$.\end{defn}

\begin{defn} We define a \emph{$G$-graded finitary category} to be an additive idempotent complete category enriched over $\bbk$-$\Mod_{G,0}^{\operatorname{gf}}$ with a finite set of isomorphism classes of indecomposable objects.\end{defn}

When $G=\{e\}$ is the trivial group, the above definition is precisely that of a finitary category first defined in \cite{mazorchuk2011cell}.

\begin{defn} Let $\ceC$ be a $G$-graded category. We define its \emph{$G$-envelope} $\hat{\ceC}$ as a category with objects defined formally as symbols of the form $X\llbracket g\rrbracket$ where $X$ is an object of $\ceC$ and $g\in G$. For notational convenience we set $X\llbracket g\rrbracket\llbracket h\rrbracket:=X\llbracket g+h\rrbracket$. We set hom-spaces as $$\Hom_{\hat{\ceC}}(X\llbracket g\rrbracket,Y\llbracket h\rrbracket)\cong\Hom_{\ceC}(X,Y)\llbracket h-g\rrbracket$$ with composition of morphisms is given by the obvious inheritance from $\ceC$. If $\ceC$ is a $G$-graded finitary category then we call $\hat{\ceC}$ a \emph{$G$-finitary category}.\end{defn}

We denote by $\id_{X,g}:X\to X\llbracket g\rrbracket$ the canonical isomorphism for any object $X\in \ceC$ and any $g\in G$, which is homogeneous of degree $-g$ with inverse $\id_{X\llbracket g\rrbracket,-g}$.

\begin{defn} Let $\ceC$ and $\ceD$ be $G$-finitary categories. A functor $F:\ceC\to\ceD$ is a \emph{$G$-graded} functor if it respects the structure of the grading and the envelope. Explicitly, this means the following:\begin{itemize}
		\item For $X$ an object of $\ceC$ and $g\in G$, $F(X\llbracket g\rrbracket)=F(X)\llbracket g\rrbracket$.
		\item For $X$ and $Y$ objects in $\ceC$, $F_{X,Y}:\Hom_\cC(X,Y)\to\Hom_\cD(FX,FY)$ is homogeneous of degree zero; that is, $\deg(F(\alpha))=\deg(\alpha)$ for any homogeneous morphism $\alpha$.\end{itemize}\end{defn}

\begin{defn} Let $\bbk$-$\cCat_G$ denote the category whose objects are $G$-graded categories and whose morphisms are all $G$-graded functors between them. Let $\bbk$-$\cCat_G^{\operatorname{gf}}$ denote the category whose objects are $G$-graded finitary categories and whose morphisms are all $G$-graded functors between them.\end{defn}

\begin{defn} We define a \emph{$G$-graded 2-category} as a category enriched over $\bbk$-$\cCat_G$. Explicitly, it has $G$-graded hom-spaces of 2-morphisms such that horizontal and vertical composition both respect degree. We define a \emph{locally $G$-graded finitary 2-category} to be a category with countably many objects enriched over $\bbk$-$\cCat_G^{\operatorname{gf}}$ such that each identity 1-morphism is indecomposable.\end{defn}

\begin{defn} Let $\cC$ be a $G$-graded 2-category. We define the \emph{$G$-envelope} 2-category $\hat{\cC}$ of $\cC$ by taking the same objects as $\cC$, and defining each hom-category $\hat{\cC}(\tti,\ttj)$ as the $G$-envelope of the category $\cC(\tti,\ttj)$. We further require that composition respects the envelope; that is, for 1-morphisms $X\llbracket g\rrbracket$ and $Y\llbracket h\rrbracket$, $X\llbracket g\rrbracket\circ Y\llbracket h\rrbracket=(X\circ Y)\llbracket g+h\rrbracket$ wherever this makes sense. We also define horizontal and vertical composition of 2-morphisms as the obvious inheritance from composition in $\cC$. If $\cC$ is a locally $G$-graded finitary 2-category, then we say that its $G$-envelope is a \emph{locally $G$-finitary 2-category}.\end{defn}

Again, if we take $G=\{e\}$ to be trivial, a locally $G$-graded finitary or locally $G$-finitary 2-category is just a locally finitary 2-category in the sense of the previous sections.

\begin{defn} Let $\cC$ be a locally $G$-finitary 2-category. If we have a weak object preserving anti-autoequivalence ${-}^*$ such that for any 1-morphism $X\in\cC(\tti,\ttj)$ has natural homogeneous 2-morphisms $\alpha:X\circ X^*\to\bbon_\ttj$ and $\beta:\bbon_\tti\to X^*\circ X$ of degree zero such that $(\alpha\circ_H\id_X)\circ_V (\id_X\circ_H\beta)=\id_X$ and $(\id_{X^*}\circ_H\alpha)\circ_V(\beta\circ_H\id_{X^*})=\id_{X^*}$, then we say that $\cC$ is a \emph{locally weakly $G$-fiat} 2-category. If $-^*$ is a weak involution, we say that $\cC$ is \emph{locally $G$-fiat}.\end{defn}

\subsection{2-Representations And Ideals}\label{Sec2RepIdeal}

We first examine functors for the graded setup, again following \cite{brundan2017categorical}.

\begin{defn} Let $\cC$ and $\cB$ be locally $G$-finitary 2-categories. We say that a strict 2-functor $F:\cC\to\cB$ is a \emph{$G$-graded} strict 2-functor if each component functor $F_{\tti,\ttj}:\cC(\tti,\ttj)\to\cB(F\tti,F\ttj)$ is a $G$-graded functor.\end{defn}

\begin{defn} Let $\cC$ and $\cB$ be locally $G$-finitary 2-categories, with strict $G$-graded 2-functors $P,Q:\cC\to\cB$. We define a \emph{$G$-graded} 2-natural transformation as a 2-natural transformation $\alpha:P\to Q$ such that for each 1-morphism $X\in\cC$, the associated 2-morphism $\alpha_X$ is of degree zero.\end{defn}

We denote by $\LA_\bbk^{G\text{-gf}}$ the 2-category which has as objects $G$-finitary categories, as 1-morphisms $\bbk$-linear additive $G$-graded functors, and as 2-morphisms natural transformations of these. We will also be using the 2-category $\LR_\bbk$ as defined in \autoref{TarCats}.

We now recall the definition of a 2-representation from \cite{mazorchuk2011cell}, and give its specification to this setup.

\begin{defn}\label{Gr2RepDefn} Let $\cC$ be a locally $G$-finitary 2-category. We define a \emph{$G$-finitary} 2-representation to be a strict $G$-graded 2-functor from $\cC$ to $\LA_\bbk^{G\text{-gf}}$. An \emph{abelian} 2-representation is a strict 2-functor from $\cC$ to $\LR_\bbk$.
 \end{defn}

We retain the notation for 2-representations, principal 2-representations etc. from the locally finitary case. There are other concepts that we have previously defined that still apply to our case - we retain the concepts of $\cL$-, $\cR$-, $\cD$- and $\cJ$-orders and cells, as well as the concept of strongly-regular cells. Ideals in 2-representations and 2-ideals also apply with no issues.

\subsection{Degree Zero Sub-2-Categories}\label{SeccC0}

We wish to generalise the idea of a coalgebra 1-morphism in the 2-category and the related theory to locally $G$-finitary 2-categories. However, in general the method of abelianisation given in \cite{mackaay2016simple} is not guaranteed to give an abelian category. Explicitly, as was shown in \cite{freyd1966representations}, the process of injective (respectively projective) abelianisation given in \cite[Section 3]{mackaay2016simple} results in an abelian (2-)category if and only if the original (2-)category has weak kernels (respectively weak cokernels).\par

We instead consider locally \emph{restricted} $G$-finitary 2-categories; that is, locally $G$-finitary 2-categories where the hom-spaces of 2-morphisms are not only graded-finite dimensional, but also finite dimensional in totality. In this case the 2-categories are simply locally finitary 2-categories, but with extra structure on the hom-spaces of 2-morphisms.\par

Let $\ceC$ be a $G$-finitary category. We define a subcategory $\ceC_0$ by taking the objects of $\ceC_0$ to be the same as the objects of $\ceC$ but taking the morphisms to be only those morphisms of $\ceC$ that are homogeneous of degree zero. Let $\cC$ be a locally $G$-finitary 2-category with a $G$-finitary 2-representation $\mbM$. We define a sub-2-category $\cC_0$ to have the same objects as $\cC$, and we set the hom-categories to be $\cC_0(\tti,\ttj)=(\cC(\tti,\ttj))_0$ for all objects $\tti,\ttj\in\cC$. We note that this implies that the 1-morphisms of $\cC_0$ are also the same as those of $\cC$. Further, it is still the case that $\bbon_\tti$ is an indecomposable 1-morphism for each object $\tti\in\cC$. However $\cC_0$ is \emph{not} a locally finitary 2-category - since we can no longer guarantee that $F\cong F\llbracket g\rrbracket$ for any non-zero $g$ as the canonical isomorphism $\id_{F,g}$ is of non-zero degree, in general $\ceC_0$ has infinitely many isomorphism classes of indecomposable 1-morphisms. This will turn out to be a surmountable problem.\par

Given a $G$-finitary 2-representation $\mbM$ of $\cC$, we define the 2-representation $\mbM_0$ of $\cC_0$ to by setting $\mbM_0(\tti)=(\mbM(\tti))_0$ for all objects $\tti$ of $\cC_0$. This can be naturally viewed as a 2-representation of $\cC_0$. Given a 1-morphism $F\llbracket g\rrbracket\in\cC_0$, we define the functor $\mbM_0(F\llbracket g\rrbracket) $ as the restriction of $\mbM(F\llbracket g\rrbracket)$ to $\mbM_0(\tti)$. This can be done since by the definition of a $G$-finitary 2-representation $$\mbM(F\llbracket g\rrbracket)_{M,N}:\Hom_{\mbM}(M,N)\to\Hom_\mbM(FM\llbracket g\rrbracket,FN\llbracket g\rrbracket)$$ is a homogeneous map of degree zero, and thus restricts to a morphism between the degree zero subspaces. Further, as horizontal and vertical composition in $\cC$ are also defined to preserve degree, the restriction of $\mbM(\alpha)$ for $\alpha:F\to G$ a homogeneous 2-morphism of degree zero to $\mbM_0(\alpha):\mbM_0(F)\to\mbM_0(G)$ is also well defined. We also notate $\ceM_0:=\coprod_{\tti\in\cC_0} \mbM(\tti)_0$ in a similar fashion to $\ceM$.\par

\begin{psn} Assume that $\mbM$ is transitive. The 2-representation $\mbM_0$ of $\cC_0$ is also a transitive 2-representation.
	\begin{proof} Let $M,N\in\ceM_0$ with $N$ indecomposable. It is sufficient for the proof to find $F\in\cC$ such that $N$ is isomorphic to a summand of $FM$ in $\ceM_0$, i.e. via an isomorphism that is homogeneous of degree zero. Since $\mbM$ is transitive, we have some $\bar{F}\in\cC$ such that there exists an isomorphism $\bar{\varphi}:\bar{F}M\to N\oplus N''$ in $\ceM$ for some $N''\in\ceM$. We therefore have morphisms $\iota:N\to\bar{F}M$ and $\sigma:\bar{F}M\to N$ such that $\sigma\iota=\id_N$.\par
		
		Setting the homogeneous decompositions $\iota=\sum_{g\in G} \iota^g$ and $\sigma=\sum_{g\in G} \sigma^g$, we see from comparison of degree with $\id_N$ that for $g\neq 0$, $\sum_{h\in G}\sigma^h\iota^{g-h}=0$ while $\sum_{h\in G} \sigma^h\iota^{-h}=\id_N$. Since $N$ is indecomposable, by standard nilpotent arguments there exists a $g\in G$ such that $\sigma^g\iota^{-g}$ is an automorphism.\par
		
		We thus have some $\rho\in\End_\ceM(N)$ such that $\rho\sigma^g\iota^{-g}=\id_N$. But as $\id_N$ and $\sigma^g\iota^{-g}$ are homogeneous of degree zero, so too must $\rho$ be. Thus we have homogeneous morphisms $\iota^{-g}:N\to \bar{F}M$ and $\rho\sigma^g:\bar{F}M\to N$ in $\ceM$ such that $\rho\sigma^g\iota^{-g}=\id_N$. We now set $F=\bar{F}\llbracket g\rrbracket$. We thus have the corresponding morphisms $\iota_g^{-g}:N\to FM$ and $\rho\sigma_{-g}^g:FM\to N$ that are homogeneous of degree zero and such that $\rho\sigma_{-g}^g\iota_g^{-g}=\id_N$. The result follows.
\end{proof}\end{psn}

\begin{psn} Let $\ceC$ be a restricted $G$-finitary category. Then $\ceC_0$ has weak kernels and weak cokernels.
	\begin{proof} Let $p:X\to Y$ be a morphism in $\ceC_0$. Consider the full subcategory $\ceC_{0,p}$ of $\ceC_0$ closed under isomorphisms and generated by $$\operatorname{add}\{X,Y,H\llbracket g\rrbracket\, |\, H\in\ceC_0\text{ indecomposable}, g\in G, \Hom_{\ceC_0}(Y,H\llbracket g\rrbracket)\neq 0\}.$$ Since the total dimension of hom-spaces in $\ceC$ is finite, given any indecomposable $H\in\ceC$ there are only finitely many $g\in G$ such that $\Hom_{\ceC}^g(Y,H)\neq 0$, and hence only finitely many $g\in G$ such that $\Hom_{\ceC_0}(Y,H\llbracket g\rrbracket)\neq 0$. Since $\ceC$ has only finitely many $G$-orbits of isomorphism classes of indecomposables and since $X$ and $Y$ each have only finitely many indecomposable summands, $\ceC_{0,p}$ contains only finitely many isomorphism classes of indecomposable 1-morphisms. It is additive and idempotent complete by construction, and as a subcategory of a category with finite dimensional hom-spaces it also has finite dimensional hom-spaces. Since it also inherits being $\bbk$-linear, $\ceC_{0,p}$ is actually a finitary category.\par
		
		There thus exists a weak cokernel $\operatorname{wcoker} p:Y\to L$ of $p$ in $\ceC_{0,p}$. We claim that $\operatorname{wcoker} p$ is a weak cokernel of $p$ in $\ceC_0$. For let $m:Y\to K$ be a morphism in $\ceC_0$ such that $mp=0$. If $m=0$, then we clearly have the zero morphism $L\to K$ satisfying the weak cokernel diagram. If $m\neq 0$, then by the definition of $\ceC_{0,p}$ we have $m\in\ceC_{0,p}$. Thus as $\operatorname{wcoker} p$ is a weak cokernel in $\ceC_{0,p}$, we have a morphism $q:L\to K$ satisfying the weak cokernel diagram, which also satisfies the diagram in $\ceC_0$. Hence $p$ does indeed have a weak cokernel and thus $\ceC_0$ has weak cokernels. The weak kernel case is precisely dual to the above argument, and the result follows.\end{proof}\end{psn}

\begin{cor}\label{GradAbelian} Let $\cC$ be a locally restricted $G$-finitary 2-category. Then $\cC_0$ has weak cokernel 2-morphisms and weak kernel 2-morphisms.\end{cor}

\begin{cor} Let $\cC$ be a locally restricted $G$-finitary 2-category with a $G$-finitary 2-representation $\mbM$. Let $\ceM=\coprod_{\tti\in\cC}\mbM(\tti)$. Then $\ul{\cC_0}$ is an abelian 2-category and $\ul{\ceM_0}$ is an abelian category.

\begin{proof} This is a direct consequence of applying the preproof to \cite[Theorem 4]{freyd1964abelian} to \autoref{GradAbelian}.\end{proof}\end{cor}

As a cautionary note, the author knows of no reasonable way to give $\ul{\cC}$ or $\ul{\mbM}$ a graded structure that is compatible with the gradings on $\cC$ and $\mbM$. This is why abelian 2-representations of locally $G$-finitary 2-categories are ungraded in \autoref{Gr2RepDefn}, and why various results below refer to component 2-morphisms of 2-morphisms in the abelianisation, rather than the 2-morphisms themselves.

\subsection{Grading Coalgebras}

For this section let $\cC$ be a locally restricted $G$-finitary 2-category and let $\mbM$ be a $G$-finitary 2-representation of $\cC$. Choose $T\in\mbM(\ttj)$ and $S\in\mbM(\tti)$. Following \cite{mackaay2016simple} and \autoref{MMMTGen} we construct a functor (which we denote $\Gamma$) from $\cC(\tti,\ttj)$ to $\bbk$-$\modd$ which takes $F$ to $\Hom_\ceM(T,FS)$, and a functor $\Gamma_0$ from $\cC_0(\tti,\ttj)$ to $\bbk$-$\modd$ which takes $F$ to $\Hom_{\ceM_0}(T,FS)$. We can extend these uniquely to left-exact functors $\ul{\Gamma}$ and $\ul{\Gamma_0}$ from $\ul{\cC}(\tti,\ttj)$ and $\ul{\cC_0}(\tti,\ttj)$ respectively to $\bbk$-$\modd$.\par

We now introduce the equivalent of the representative 1-morphisms in \cite[Section 4.1]{mackaay2016simple}. Since $\ul{\Gamma_0}$ is left exact, by \cite[Section 1.8]{grothendieck1971revetements}, it is pro-representable; that is, a small filtered colimit of representable functors. However by definition $\ul{\cC_0}(\tti,\ttj)$ has enough injectives and the functor category is closed under small filtered colimits, and thus the functor is in fact representable. We denote this representative by $[S,T]_0$.\par

We have the following analogy of \cite[Lemma 4.2]{mackaay2016simple} that we will find useful:

\begin{lem}\label{Repble0} For any $H\in\coprod_{\tti\in\cC_0}\ul{\cC_0}(\tti,\ttj)$,  $$\Hom_{\ul{\mbM_0}}(T,HS)\cong\Hom_{\ul{\cC_0}(\tti,\ttj)}([S,T]_0,H)$$ in $\bbk$-$\modd$.
	\begin{proof} Take $H\in\ul{\cC_0}(\tti,\ttj)$. Then $H$ has an exact sequence $H\hookrightarrow F_1\to F_2$ where $F_i\in\cC_0(\tti,\ttj)$ for both $i$. We have left exact functors $\Hom_{\ul{\mbM_0}}(T,-S)$ and $\Hom_{\ul{\cC}}([S,T]_0,-)$ and a diagram\newline
		\xymatrix{\Hom_{\ul{\mbM_0}}(T,HS) \ar@{^{(}->}[d] & \Hom_{\ul{\cC_0}}([S,T]_0,H) \ar@{^{(}->}[d] \\ \Hom_{\mbM_0}(T,F_1S)\ar[d] \ar[r]^-\sim & \Hom_{\ul{\cC_0}}([S,T]_0,F_1) \ar[d] \\ \Hom_{\mbM_0}(T,F_2S) \ar[r]^-\sim & \Hom_{\ul{\cC_0}}([S,T]_0,F_2) }\newline of vector spaces. The result then follows from an application of the five lemma.
\end{proof}\end{lem}

The functor $\ul{\Gamma}$ is representable by \cite{mackaay2016simple}, with representative 1-morphism $[S,T]$.

\begin{psn}\label{ACoalga} If $S=T$, then $\tti=\ttj$ and $A^S=[S,S]$ has the structure of a coalgebra 1-morphism in $\ul{\cC}(\tti,\tti)$ and $A_0^S=[S,S]_0$ has the structure of a coalgebra 1-morphism in $\ul{\cC_0}(\tti,\tti)$.
	\begin{proof} The first result is \autoref{LFMMMTL5} and the second result is mutatis mutandis the first.
		
\end{proof}\end{psn}

\subsection{The Main Results}
Let $\cC$, $\cC_0$, $A^S$ and $A_0^S$ be as in the previous section. The reason we wish to study $A_0^S$ is that we do not a priori know anything about the component degrees of the internal 2-morphisms for the counit and comultiplication 2-morphisms for $A^S$. However, we know by definition the corresponding component 2-morphisms for $A_0^S$ have to be homogeneous of degree zero. We will show that we can take $A^S$ to be precisely $A_0^S$ with its counit and comultiplication 2-morphisms. We use the natural inclusion of $\cC_0$ into $\cC$ to state the following lemma:

\begin{lem}\label{cC0Sum} For any $F\in\cC(\tti,\tti)$ and $H\in\ul{\cC_0}(\tti,\tti)$, in $\bbk$-$\modd$ $$\Hom_{\ul{\cC}}(H,F)\cong\bigoplus_{g\in G}\Hom_{\ul{\cC_0}}(H,F\llbracket g\rrbracket).$$
	\begin{proof} Let $H=(X,k,Y_i,\alpha_i)$, using the injective fan Freyd abelianisation notation from \autoref{fincelldef}. Then as $F=(F,0,0,0)$, a morphism from $H$ to $F$ in $\ul{\cC}$ is of the form $[(p,0)]$ with $p:X\to F$ a morphism in $\cC$ and the equivalence relation is spanned by those $p$ such that there exist $q_i:Y_i\to F$ with $\sum q_i\alpha_i=p$. We let $p=\sum_{g\in G} p_g$ for some finite sum. Assume that $[(\tilde{p},0)]$ is equivalent to $[0]$ with morphisms $q_i$ as specified. Since $H\in\ul{\cC_0}(\tti,\tti)$, the $\alpha_i$ are homogeneous of degree zero. In particular, if we decompose each $q_i$ as $q_i=\sum_{h\in G} q_{i,h}$, then by comparison of degree if follows that $\tilde{p}_g=\sum_i q_{i,g}\alpha_i$. Therefore if $[(p,0)]$ is a general morphism from $H$ to $F$, then $[(p,0)]=\sum_{g\in G} [(p_g,0)]$ where the equivalence relation is spanned by those $p_g$ such that there exist $q_{g,i}:Y_i\to F$ homogeneous of degree $g$ such that $p_g=\sum_i q_{i,g}\alpha_i$.\par
		
		We thus define the map $$\bigoplus_{g\in G}\Hom_{\ul{\cC_0}}(H,F\llbracket g\rrbracket)\to\Hom_{\ul{\cC}}(H,F)$$ by taking $\sum_{g\in G} [(p_g,0)]$ to $[(\sum_{g\in G} p_g,0)]$. This is clearly a vector space homomorphism, and the above working shows that this map is well-defined and injective. If we have some $[(p,0)]\in\Hom_{\ul{\cC}}(H,F)$, then as $[(p,0)]=\sum_{g\in G}[(p_g,0)]$, and as any $p_g:H\to F$ corresponds to a degree zero 2-morphism $p_g:H\to F\llbracket g\rrbracket$, the assocaited 2-morphism $\sum_{g\in G}[(p_g,0)]\in\bigoplus_{g\in G}\Hom_{\ul{\cC_0}}(H,F\llbracket g\rrbracket)$ maps to $[(p,0)]$, the map is surjective and we have the required isomorphism.\end{proof}\end{lem}

We can now give one of the main results of the paper.

\begin{thm}\label{AN0EqAN} $A^S\cong A^S_0$ in $\ul{\cC}$. In addition, we can choose a representative of the isomorphism class of $A^S$ in $\ul{\cC}$ such that, when considered as a coalgebra 1-morphism, its coalgebra and comultiplication 2-morphisms have components homogeneous of degree zero.
	\begin{proof} For $F\in\cC(\tti,\tti)$ or $\cC_0(\tti,\tti)$, \newline
		$$\Hom_{\ul{\mbM}}(S,FS)\cong\Hom_{\mbM}(S,FS)$$ and similarly $$\Hom_{\ul{\mbM_0}}(S,FS)\cong\Hom_{\mbM_0}(S,FS).$$ By the definition of the grading on $\mbM$, $$\Hom_{\mbM}(S,FS)\cong\bigoplus_{g\in G}\Hom_{\mbM_0}(S,F\llbracket g\rrbracket S).$$ But by the definition of $A^S_0$, $$\bigoplus_{g\in G}\Hom_{\mbM_0}(S,F\llbracket g\rrbracket S)\cong\bigoplus_{g\in G} \Hom_{\ul{\cC_0}}(A^S_0,F\llbracket g\rrbracket).$$ Applying \autoref{cC0Sum} for $H=A^S_0$, we have that $\Hom_{\mbM}(S,FS)\cong\Hom_{\ul{\cC}}(A^S_0,F)$ for all $F\in\cC(\tti,\tti)$. But by the definition of the representative, $$\Hom_{\mbM}(S,FS)\cong\Hom_{\ul{\cC}}(A^S,F)$$ and $A^S$ is unique with this property up to isomorphism, hence $A^S\cong A^S_0$ as required.\par Taking $A^S_0$ of the isomorphism class of $A^S$, we have the following diagram of vector spaces: \newline \xymatrix{\Hom_{\mbM_0}(S,S)\ar[rr]^{\sim} \ar[d] & & \Hom_{\ul{\cC_0}}(A^S_0,\bbon_\tti) \ar[d] \\ \Hom_{\mbM}(S,S) \ar[rr]^{\sim} & & \Hom_{\ul{\cC}}(A^S_0,\bbon_\tti)}\newline where the vertical arrows are the natural inclusions and the horizontal arrows are the representation isomorphisms. By choice of $A^S_0$, this diagram is strictly commutative. Taking $\id_S$ in $\Hom_{\mbM_0}(S,S)$, its image along the top path is the image of $\epsilon_0$ under the natural injection (i.e. $\epsilon_0$ considered as a 2-morphism in $\ul{\cC}$) while the image under the lower path is the counit of $A^S_0$ as a coalgebra in $\ul{\cC}$. It follows that this counit is equal to $\epsilon_0$, and thus has components (or more accurately, non-zero component) homogeneous of degree zero. By constructing similar diagrams for the coevaluation and hence comultiplication 2-morphisms, the second claim follows.\end{proof}\end{thm}

\subsection{An Application: 2-Kac--Moody Algebras}\label{gr2KM}

We return to considering the locally fiat 2-category $\cU_\Lambda$ found in \autoref{2KMADefn}. The following is an immediate consequence of the definition of a cyclotomic KLR algebra:

\begin{psn} $\cU_\Lambda$ is a locally (restricted) $\bbZ$-finitary 2-category.\end{psn}

We set $R=\bigoplus_{n\geq 0} R^\Lambda(n)$, where $R^\Lambda(n)$ is the cyclotomic KLR algebra of degree $n$.

\begin{psn}\label{2KMResTopJ} The indecomposable 1-morphisms of the form $Q_1\bbon_\Lambda Q_2$, where the $Q_i$ are 1-morphisms in $\cU_\Lambda$, form a maximal $\cJ$-cell in $\cU_\Lambda$.
	\begin{proof} Since the hom-categories of $\cU_\Lambda$ are idempotent complete, an indecomposable 1-morphism of the form $Q_1\bbon_\Lambda Q_2$ is isomorphic to a direct summand of a functor of the form $M_1 e_1\otimes_\bbk e_2 M_2\otimes_ {R^\Lambda}-$, where the $e_i$ are primitive idempotents in $R^\Lambda$ and $M_1$ and $M_2$ are products of the $\mathfrak{e}_i$ and $\mathfrak{f}_i$. In particular, we have that $M_1 e_1=R^\Lambda e_1$ and $e_2 M_2=e_2 R^\Lambda$. It follows that such a bimodule is a projective $R^\Lambda$-$R^\Lambda$-bimodule.\par
		
		Similarly, given some 1-morphism $Q$ that corresponds to a functor $M\otimes_{R^\Lambda}-$, we can choose primitive idempotents $e$ and $f$ of $R^\Lambda$ such that $eMf\neq 0$. Then for any $R^\Lambda e'\otimes_\bbk f'R^\Lambda$ ($e'$ and $f'$ primitive), $$R^\Lambda e'\otimes_\bbk eR^\Lambda\otimes_{R^\Lambda} M\otimes_{R^\Lambda} R^\Lambda f\otimes_\bbk f'R^\Lambda\cong (R^\Lambda e'\otimes_\bbk f'R^\Lambda)^{\oplus m},$$ where $m=\dim eMf$. Thus $M\otimes_{R^\Lambda}-\leq_\cJ R^\Lambda e'\otimes_\bbk f'R^\Lambda\otimes_{R^\Lambda} -$. This shows in particular that any indecomposable 1-morphism isomorphic to a summand of a functor of the form $R^\Lambda e\otimes_\bbk fR^\Lambda\otimes_{R^\Lambda}-$ (for $e$ and $f$ primitive) is $\cJ$-equivalent to any other 1-morphism isomorphic to a functor of the same form.
		
		It is immediate from the previous paragraph that the $\cJ$-cell containing (the indecomposable summands of) functors of the form $R^\Lambda e\otimes_\bbk fR^\Lambda\otimes_{R^\Lambda}-$, for $e$ and $f$ primitive, is maximal, and it remains to show that these functors exhaust the isomorphism classes of members of the $\cJ$-cell. By construction $e(\beta, i) R^\Lambda(\beta+\alpha_i)$ is a projective right $R^\Lambda(\beta+\alpha_i)$-module (and hence a projective right $R^\Lambda$-module), while by \cite[Theorem 4.5]{kang2012categorification} it is a projective left $R^\Lambda(\beta)$-module (and hence a projective left $R^\Lambda$-module). A similar argument gives that $R^\Lambda(\beta+\alpha_i)e(\beta,i)$ is a projective left and projective right $R^\Lambda$-module. As a consequence, every 1-morphism in $\cU_\Lambda$ is both left projective and right projective.\par
		
		Let $M$ be some $(R^\Lambda$-$R^\Lambda)$-bimodule with $Q=M\otimes_{R^\Lambda}-$ such that there exist primitive idempotents $e$ and $f$ with $Q\geq_\cJ R^\Lambda e\otimes_\bbk fR^\Lambda\otimes_{R^\Lambda}-$ in $\cU_\Lambda$. This means that there are some 1-morphisms $T\otimes_{R^\Lambda}-$ and $S\otimes_{R^\Lambda}-$ in $\cU_\Lambda$ such that $M$ is a direct summand of $$T\otimes_{R^\Lambda} R^\Lambda e\otimes_\bbk fR^\Lambda\otimes_{R^\Lambda} S\cong Te\otimes_\bbk fS.$$ Since $T$ is left projective and $S$ is right projective, $M$ thus decomposes over $\bbk$ and is a summand of a bimodule of the form $R^\Lambda e'\otimes_\bbk f'R^\Lambda$ for some primitive idempotents $e'$ and $f'$. Both of the remaining claims follow immediately, and the result is proved.\end{proof}\end{psn}

	\begin{lem}\label{2KMGrCell} Every cell 2-representation of $\cU_\Lambda$ is a graded simple transitive 2-representation.
		\begin{proof} Let $\ceJ$ be a $\cJ$-cell in $\cU_\Lambda$. By considering the 2-category $\cU_{\Lambda,\cJ}$ as defined in \autoref{JSimpSuff}, without loss of generality $\ceJ=:\ceJ_\Lambda$ is the highest $\cJ$-cell of $\cU_\Lambda$, which by \autoref{2KMResTopJ} is the indecomposable 1-morphisms that factor over $\Lambda$. Since these all correspond to tensoring with a projective $(R^\Lambda$-$R^\Lambda)$-bimodule, we can embed $\ceJ_\Lambda$ into the 2-category $\cC_R$ associated to $R^\Lambda$ (c.f. \autoref{cCA}; since the $R^\Lambda(\beta)$ are not necessarily basic, see specifically the definition at the end of that section). In fact, we claim that this embedding is an equivalence between $\cC_R$ and $\cU_{\Lambda,\ceJ}$.\par
			
			To see this, by \cite[Theorem 4.4]{varagnolo2011canonical} $\ceJ$ contains $R^\Lambda e\otimes_\bbk\bbk$ for all primitive idempotents $e$ of $R^\Lambda$. But then as $\ceJ$ is strongly regular, it is closed under adjunctions and hence $\bbk\otimes_\bbk fR^\Lambda$ for all primitive idempotents $f$. But then as $\ceJ$ is closed under direct summands of compositions, it also contains $R^\Lambda e\otimes_\bbk fR^\Lambda$ for all primitive idempotents $e$ and $f$. Therefore $\ceJ$ not only embeds into $\cC_R$, but also essentially surjects, giving the required equivalence.\par
			
			This means that any $\cL$-cell of $\ceJ_\Lambda$ will give an equivalent cell 2-representation by \autoref{celldef}, and we choose a particularly useful one. Consider the $\cL$-cell $\cL_\bbk$ which embeds into $\cC_R$ as (the finite dimensional elements of) $\add\{R^\Lambda\otimes_\bbk\bbk\}$. To construct the cell 2-representation, we first construct the transitive (but not necessarily simple transitive) 2-representation $\mbN_{\cL_\bbk}$ (c.f. \autoref{cCA}). This is a graded 2-representation by construction, and thus to show the cell 2-representation is graded it suffices to show that the ideal $\cI$ of the 2-representation we quotient by to form the simple transitive quotient is homogeneous. But given some indecomposable $R^\Lambda e\otimes_\bbk \bbk$ for some idempotent $e$, we recall from \autoref{CAXCellStr} that $\cI$ is generated by morphisms of the form $\varphi_{a,b}:R^\Lambda e\otimes_\bbk\bbk\to R^\Lambda e\otimes_\bbk \bbk$ where $\varphi_{a,b}(e\otimes 1)=eae\otimes b$, with $b\in\rad \bbk $. But $\rad \bbk =0$, and hence $\cI=0$, which is trivially a homogeneous ideal. The result follows.\end{proof}\end{lem}
	
	This gives us the following result:
	
	\begin{thm}\label{Gr2KMClass} Any simple transitive 2-representation of $\cU_\Lambda$ is in fact a graded 2-representation, and is equivalent to a cell 2-representation.
		\begin{proof} This is a direct consequence of combining \autoref{2KMGrCell} and \autoref{2KMMain}.\end{proof}\end{thm}

\bibliography{C:/Users/jamac/Documents/Maths/Thesis/MasterBib}
\end{document}